\documentclass[twocolumn]{autart}    

\usepackage{graphicx}          

\usepackage[titletoc,title]{appendix}

\usepackage{amsmath,amsfonts,mathrsfs}
\usepackage{stmaryrd}
\usepackage{enumerate}
\usepackage{color}
\usepackage{xcolor}
\usepackage{amsmath}
\allowdisplaybreaks
\usepackage{subfigure}

\newtheorem{theorem}{Theorem}[section]
\newtheorem{lemma}[theorem]{Lemma}
\newtheorem{definition}[theorem]{Definition}
\newtheorem{proposition}[theorem]{Proposition}
\newtheorem{corollary}[theorem]{Corollary}

\theoremstyle{definition}

\newenvironment{proof}[1][Proof]
{\par\noindent\textit{#1.}~}
{\hspace*{\fill}$\Box$\par}

\newcommand{\esssup}{\mathop{\mathrm{esssup}}}

\begin{document}
	
	\begin{frontmatter}
		
		\title{Linear-Quadratic Zero-Sum Stochastic Differential Game with Partial Observation\thanksref{footnoteinfo}} 
		
		\thanks[footnoteinfo]{This paper was not presented at any conference. Corresponding author: Wanying Yue.}
		
		\author[ZYY]{Zhiyong Yu}\ead{yuzhiyong@sdu.edu.cn},    
		\author[WYY]{Wanying Yue}\ead{yuewanying@mail.sdu.edu.cn}
		
		\address[ZYY]{School of Mathematics, Shandong University, Jinan 250100, China}  
		\address[WYY]{School of Mathematics, Shandong University, Jinan 250100, China}


\begin{abstract} 
This paper is concerned with a kind of linear-quadratic (LQ, for short) two-person zero-sum stochastic differential game problems with partial observation. We propose the notions of explicit and implicit feedback laws under partial observation. With the help of a class of conditional mean-field stochastic differential equations (CMF-SDEs, for short), the separation principle, filtering techniques, and the method of completion of squares, we construct a saddle point in the form of feedback laws for the two players. Finally, the theoretical results are applied to investigate a duopoly competition problem with partial observation.    
\end{abstract}
		
\end{frontmatter}

\section{Introduction}\label{Sec:Int}

Two-person zero-sum stochastic differential games can be seen as  decision-making problems between two opposing players in stochastic environments. The excellent mathematical structure and accurate formulation of bargaining between buyers and sellers, competition between duopoly enterprises, and other behaviors in real world endow them with both theoretical and practical significance. The non-anticipative strategy proposed by Elliott and Kalton \cite{EK-72} to characterize the influence of the opponent's actions on a player's decision is one of the important concepts in theoretical research. Based on it, Fleming and Souganidis \cite{FS-89} proved that the upper and lower value functions of a zero-sum stochastic differential game satisfy the relevant Hamilton-Jacobi-Bellman-Isaacs equations as viscosity solutions (see Crandall et al. \cite{CIL-92} for example). Later, Buckdahn and Li \cite{BL-08} further developed the work of \cite{FS-89} to the case of stochastic recursive utility. Inspired by the non-anticipative strategy, Yu \cite{Y-15} and Tian et al. \cite{TYZh-20} introduced the notions of explicit and implicit feedback laws (referred to as explicit and implicit strategy laws in \cite{TYZh-20}). In the direction of the LQ problem, they constructed saddle-point strategies in the form of feedback laws for players with the help of the associated Riccati equation.

We note that the aforementioned works are all under the assumption of complete information, that is, all the details of the systems can be completely observed by the players. This assumption is too idealistic and rarely exists in the real world. 
Therefore, stochastic differential games with partial observations are of great practical significance.

Motivated by this, we study a class of partially observed LQ zero-sum stochastic differential game problems in this paper. Specifically, find a saddle point which is in the form of feedback laws and characterizes the trade-off between Player 1 maximizing and Player 2 simultaneously minimizing the following objective functional:
\begin{equation}\label{J}
\begin{aligned}
& J(u(\cdot)) := \frac 1 2 \mathbb E \Big\{ \big\langle Gx^u(T),\, x^u(T) \big\rangle +2\langle g,\, x^u(T) \rangle \\
& \quad +\int_0^T \Bigg[ \left\langle \begin{pmatrix} Q(t) & S(t)^\top \\ S(t) & R(t) \end{pmatrix} \begin{pmatrix} x^u(t) \\ u(t) \end{pmatrix},\, \begin{pmatrix} x^u(t) \\ u(t) \end{pmatrix} \right\rangle   \\
& \quad +2 \left\langle \begin{pmatrix} q(t) \\ \rho(t) \end{pmatrix},\, \begin{pmatrix} x^u(t) \\ u(t) \end{pmatrix} \right\rangle \Bigg]\, \mathrm dt \Bigg\}, 
\end{aligned}
\end{equation}
with $u(\cdot) =(u_1(\cdot)^\top, u_2(\cdot)^\top)^\top$, where the superscript $\top$ represents the transpose of a vector (or a matrix), 
and the notations
\[
\begin{aligned}
& S(\cdot) = \begin{pmatrix} S_1(\cdot) \\ S_2(\cdot) \end{pmatrix}, \quad
R(\cdot) = \begin{pmatrix} R_{11}(\cdot) & R_{21}(\cdot)^\top \\ R_{21}(\cdot) & R_{22}(\cdot) \end{pmatrix}, \\
& \rho(\cdot) = \begin{pmatrix} \rho_1(\cdot) \\ \rho_2(\cdot) \end{pmatrix},
\end{aligned}
\]
subject to the state equation:
\begin{equation}\label{Sys}
\left\{
\begin{aligned}
& \mathrm d x^u(t) = \big[ A(t)x^u(t) +B(t)u(t) +b(t) \big]\, \mathrm dt\\
& \qquad\qquad +C(t)\, \mathrm dW(t) +\bar C(t)\, \mathrm d\bar W(t),\\
& x^u(0) =a
\end{aligned}
\right.
\end{equation}
with the decomposition $B(\cdot) =(B_1(\cdot), B_2(\cdot))$ corresponding to that of $u(\cdot)$, and the observation equation:
\begin{equation}\label{Obs}
\left\{
\begin{aligned}
& \mathrm dy^u(t) = \big[ H(t)x^u(t) +h(t) \big]\, \mathrm dt +K(t)\, \mathrm dW(t),\\
& y^u(0) =0.
\end{aligned}
\right.
\end{equation}
The detailed and rigorous formulation will be provided in Section \ref{Sec:Formu}. We may regard the current work as a development of \cite{TYZh-20,Y-15} with complete information to the case of partial observation.

We adopt the terminology of Wang et al. \cite{WWX-18} to refer to situations where complete information cannot be obtained as incomplete information situations. If the information that can be obtained comes from an observation equation similar to \eqref{Obs}, the corresponding control or game problem is referred to as partially observable. Taking the partially observable LQ zero-sum stochastic differential game problem considered in this paper as an example, the partially observed information is mathematically described by the filtration $\mathbb F^u := \{ \mathcal F^u_t \}_{t\in [0,T]}$, where
\begin{equation}\label{Int:Fu}
\mathcal F^u_t := \sigma \{ y^u(s),\ 0\leq s\leq t \}.
\end{equation}
From Equations \eqref{Sys}, \eqref{Obs} and \eqref{Int:Fu}, it can be seen that, on the one hand, the control $u(\cdot)$ should be chosen to be $\mathbb F^u$-progressively measurable; on the other hand, the definition of $\mathbb F^u$ depends on the control $u(\cdot)$. In other words, in partially observable problems, there exists a circular dependency between the selection of admissible controls and the information that can be obtained, which is an intrinsic difficulty in theoretical research. To address this difficulty, three main research methods have been developed.

The first method is to modify the mathematical model from partial observation to so-called partial information, i.e., the information that can be obtained is directly given by an abstract sub-filtration $\mathbb G$ that does not depend on the control $u(\cdot)$, thereby directly removing the difficulty of circular dependency. For the study on stochastic control and game problems with partial information, please refer to \cite{BO-07,HWW-16,M-09,NWY-22,WY-12}, etc.


In the second method, by employing Girsanov's theorem, the observation process is interpreted as a standard Brownian motion on a reference probability space with a new probability measure, thereby eliminating the circular dependency in that reference probability space. The advantage of this method is that it can be applied to handle some nonlinear problems, but the disadvantage is that the mathematical treatment of transferring the problem to the reference probability space requires some strong constraints. Specifically, for the LQ problem we are concerned with, it may be necessary to assume that the drift coefficient $H(\cdot)x^u(\cdot) +h(\cdot)$ in \eqref{Obs} is bounded, which is too strong to be readily accepted. Therefore, this method is not adopted in the present paper. For the research on the Girsanov transformation method for partially observed stochastic control and game problems, please refer to \cite{BEK-89,SC-16,T-98,XZhZh-19}, etc.


The third method is the state decomposition method for linear systems, which is also the method adopted in this paper. Specifically, inspired by Bensoussan \cite{B-18}, we restrict the admissible controls to be progressively measurable with respect to $\mathbb F^u$ and $\mathbb F^0$, where $\mathbb F^0$ is the filtration generated by the observation process $y^0(\cdot)$ when the control is zero (see \eqref{Int:Fu}). Thanks to the linearity of the state equation \eqref{Sys} and the observation equation \eqref{Obs}, we have $\mathbb F^u = \mathbb F^0$ (see Lemma \ref{LEM:Fu=F0}). Therefore, the study of the problem is shifted to under the filtration $\mathbb F^0$ (which does not depend on the selection of control) to eliminate the circular dependency. For the research on the state decomposition method for partially observed stochastic optimal control problems, please see \cite{WWX-15,WXX-17,WX-24,W-68,ZhXS-21}, etc.


Compared to control problems, there is very little research on the state decomposition method for partially observable stochastic differential game problems. To our knowledge, only the work of Wu and Zhuang \cite{WZh-18} is related to our research, where they have done some studies under one-dimensional and some other strong assumptions. The reason for this phenomenon may be the interplay between the control actions of players in game problems. This interplay is involved in the aforementioned difficulty of the control-observation circular dependency, making the problem even more perplexing. Specifically, 
%
the set of admissible controls jointly owned by the players cannot be represented as the Cartesian product of the individual players' admissible control sets under general partial observation settings (see the explanation following Definition \ref{Def:Add-Con} in Section \ref{Sec:Formu}). This point is different from the situations of complete information and partial information, and is not encountered in optimal control problems with partial observation. 

In response, for our LQ zero-sum problem, we propose a notion of feedback laws under partial observation, and by integrating filtering technique, the separation principle, the complication-of-squares method, and block-matrix diagonalization technique, conjecture candidate saddle points. Then, with the help of well-posedness result for a class of CMF-SDEs and measurability-analysis technique, we verify that the candidate saddle points are true saddle points. This approach effectively overcomes the new difficulty described in the previous paragraph.


Now, we would like to summarize the innovations and features of this paper as the following three points:
\begin{itemize}

\item The notions of explicit and implicit feedback laws under complete information in \cite{Y-15} and \cite{TYZh-20} are extended to the case of partial observation. Accordingly, a class of CMF-SDEs is introduced as a research tool for our partially observable problem.


\item The two-step completion-of-squares procedure introduced in \cite{Y-15} for studying LQ zero-sum differential games is consolidated into the single-step completion-of-squares method of classical LQ optimal control, together with the block-matrix diagonalization technique.


\item For the partially observable LQ zero-sum stochastic differential game problem considered in this paper, with the help of a Riccati equation and an ordinary differential equation (ODE, for short), we provide a saddle point in the form of feedback laws for the two players. We think that this form of saddle point is appealing in practice (see the explanation following Definition \ref{Def:Sad-P}).
\end{itemize}

The rest of this paper is organized as follows. In Section \ref{Sec:Spa}, some spaces that will be used later are presented. In Section \ref{Sec:Formu}, we provide a detailed and rigorous formulation for the LQ zero-sum stochastic differential game problem with partial observation. In Section \ref{Sec:F-S}, we apply the Kalman-Bucy filter and the Wonham type separation principle to transform the problem with partial observation into a problem with complete information. Unlike \cite{Y-15} and \cite{TYZh-20}, in this problem with complete information, some additional non-homogeneous terms appear in the state equation and some additional linear terms appear in the objective functional, then the conclusions in \cite{Y-15} and \cite{TYZh-20} cannot be directly applied. In Section \ref{Sec:Com-Sad}, we first apply the method of completion of squares, with the help of a Riccati equation and an ODE, to provide candidate saddle points in the form of feedback laws. Then, using the theory of CMF-SDEs and filtering equations, we prove that the candidate saddle points are indeed true saddle points. In Section \ref{Sec:Duop}, as an application of the theoretical results, we solve an example of two-dimensional duopoly competition with partial observation and do some numerical simulations. Finally, we conclude this paper in Section \ref{Sec:Conclu}. 

\section{Spaces}\label{Sec:Spa}
	
	Let $\mathbb R^n$ be the $n$-dimensional Euclidean space equipped with the  Euclidean inner product $\langle \cdot,\, \cdot \rangle$. The induced norm is denoted by $|\cdot|$. Let $\mathbb R^{n\times m}$ be the collection of all $(n\times m)$ matrices and $\mathbb S^n \subset \mathbb R^{n\times n}$ consist of all $(n\times n)$ symmetric matrices. Clearly, both $\mathbb R^{n\times m}$ and $\mathbb S^n$ are Euclidean spaces.

	Let $T>0$ be a fixed time horizon. Let $(\Omega,\mathcal F, \mathbb F, \mathbb P)$ be a complete filtered probability space on which it is defined a $(d+\bar d)$-dimensional standard Brownian motion $(W(\cdot)^\top, \bar W(\cdot)^\top)^\top$ with $W(\cdot) =(W_1(\cdot),W_2(\cdot),\dots,W_d(\cdot))^\top$ and $\bar W(\cdot) =(\bar W_1(\cdot), \bar W_2(\cdot), \dots, \bar W_{\bar d}(\cdot))^\top$. The filtration $\mathbb F =\{ \mathcal F_t \}_{t\in [0,T]}$ is the natural one generated by $W(\cdot)$ and $\bar W(\cdot)$ and augmented by all $\mathbb P$-null sets. Let $\mathcal F = \mathcal F_T$.

	We continue to introduce some Banach (sometimes more precisely, Hilbert) spaces of random variables, deterministic processes, or stochastic processes as follows:
	\begin{itemize}
		\item $L^2_{\mathcal F_t}(\Omega;\mathbb R^n)$ is the set of $\mathcal F_t$-measurable random variables $\xi:\Omega\rightarrow \mathbb R^n$ such that $\mathbb E[|\xi|^2]<\infty$.
		
		\item $L^p(0,T;\mathbb R^n)$ (with $p=1,2$) is the set of Lebesgue measurable deterministic processes $f:[0,T] \rightarrow \mathbb R^n$ such that $\int_0^T |f(t)|^p\, \mathrm dt <\infty$.
		When $p=1$, we denote by $L(0,T;\mathbb R^n) := L^1(0,T;\mathbb R^n)$.

		\item $L^\infty(0,T;\mathbb R^n) \subset L^2(0,T;\mathbb R^n)$ consists of all elements such that $\esssup_{t\in [0,T]} |f(t)| <\infty$.

\item $L^2_{\mathbb F}(0,T;\mathbb R^n)$ is the set of $\mathbb F$-progressively measurable stochastic processes $f:\Omega \times [0,T] \rightarrow \mathbb R^n$ such that $\mathbb E \int_0^T |f(t)|^2\, \mathrm dt <\infty$.
		
		
		\item $L^2_{\mathbb F}(\Omega;C(0,T;\mathbb R^n)) \subset L^2_{\mathbb F}(0,T;\mathbb R^n)$ consists of all processes such that for almost all $\omega \in \Omega$, $t\mapsto f(\omega,t)$ is continuous and $\mathbb E[ \sup_{t\in [0,T]} |f(t)|^2 ]<\infty$.
	\end{itemize}

	\section{Problem formulation}\label{Sec:Formu}
	
In this section, we provide a detailed and rigorous formulation for the LQ zero-sum stochastic differential game problem with partial observation. Firstly, for the game system \eqref{Sys}, we assume that $a \in \mathbb R^n$, $A(\cdot) \in L^\infty(0,T;\mathbb R^{n\times n})$, $B_i(\cdot) \in L^\infty(0,T;\mathbb R^{n\times k_i})$ ($i=1,2$), $b(\cdot) \in L(0,T;\mathbb R^n)$, $C(\cdot) \in L^\infty(0,T;\mathbb R^{n\times d})$ and $\bar C(\cdot) \in L^\infty(0,T;\mathbb R^{n\times \bar d})$. Then, the basic well-posedness result of stochastic differential equations (SDEs, for short) implies that, for any $u_1(\cdot) \in L^2_{\mathbb F}(0,T;\mathbb R^{k_1})$ and any $u_2(\cdot) \in L^2_{\mathbb F}(0,T;\mathbb R^{k_2})$, SDE \eqref{Sys} admits a unique solution $x^u(\cdot) \in L^2_{\mathbb F}(\Omega;C(0,T;\mathbb R^n))$. In the terminology of control and game theory, $u_i(\cdot)$ is referred to as the control process of Player $i$ ($i=1,2$), and $x^u(\cdot)$ is referred to as the corresponding state process.

Secondly, for the objective functional \eqref{J}, we let $G \in \mathbb S^n$, $g\in \mathbb R^n$, $Q(\cdot) \in L^\infty(0,T;\mathbb S^n)$, $q(\cdot) \in L(0,T;\mathbb R^n)$, $S_i(\cdot) \in L^\infty(0,T;\mathbb R^{k_i \times n})$, $R_{ii}(\cdot) \in L^\infty(0,T;\mathbb S^{k_i})$, $R_{21}(\cdot) \in L^\infty(0,T;\mathbb R^{k_2\times k_1})$ and $\rho_i(\cdot) \in L^2(0,T;\mathbb R^{k_i})$ ($i=1,2$). Clearly, with this setting, the objective functional \eqref{J} is well-defined, i.e., $|J(u(\cdot))| <\infty$ for any $u(\cdot) \in L^2_{\mathbb F}(0,T;\mathbb R^k)$ where $k:=k_1+k_2$. In our zero-sum game, Player $1$ aims to maximize the objective functional by choosing controls; conversely, Player $2$ aims to minimize the objective functional. The objective functional can be viewed as the payoff for Player $1$ and the cost for Player $2$.

	
In numerous practical problems, the state process of a game system cannot be fully observed, while what can often be observed is a part of the state process or its transformation, which is referred to as the observation process. In this paper, the observation process, denoted as $y^u(\cdot)$, is formulated to obey the SDE \eqref{Obs}. Now, we assume that $H(\cdot) \in L^\infty(0,T;\mathbb R^{d\times n})$, $h(\cdot) \in L(0,T;\mathbb R^d)$ and $K(\cdot) \in L^\infty(0,T;\mathbb R^{d\times d})$. Clearly, since $x^u(\cdot) \in L^2_{\mathbb F}(\Omega; C(0,T;\mathbb R^n))$, SDE \eqref{Obs} also admits a unique solution $y^u(\cdot) \in L^2_{\mathbb F}(\Omega; C(0,T;\mathbb R^d))$.

For any $u(\cdot) \in L^2_{\mathbb F}(0,T;\mathbb R^k)$, let $\mathbb F^u := \{ \mathcal F^u_t \}_{t\in [0,T]}$ be the filtration generated by the observation process $y^u(\cdot)$, where $\mathcal F^u_t$ is defined by \eqref{Int:Fu}. From the viewpoint of practical application, the {\it admissible} controls for players should naturally be $\mathbb F^u$-progressively measurable, i.e., it is reasonable to choose the following set as the admissible control set:
	\begin{equation}\label{Add:bar}
		\begin{aligned}
			\bar{\mathscr U}[0,T] := \big\{ u(\cdot) \in L^2_{\mathbb F}(0,T;\mathbb R^k) \,\big|\, u(\cdot) \mbox{ is } 
			\mathbb F^u\mbox{-}\\
			\mbox{progressively measurable} \big\}.
		\end{aligned}
	\end{equation}
	However, inspired by the stochastic control problem with partial observation (see \cite{B-18,WWX-15} for example), the definition of admissible control mentioned above will bring about the so-called ``control-observation circular dependency", which is an intrinsic difficulty in the subsequent mathematical treatment. To overcome this difficulty, we adopt Bensoussan's method \cite{B-18} and modify \eqref{Add:bar} to
	\begin{equation}\label{Add:Con}
		\begin{aligned}
			\mathscr U[0,T] := \big\{ u(\cdot) \in L^2_{\mathbb F}(0,T;\mathbb R^k) \,\big|\, u(\cdot) \mbox{ is } 
			\mathbb F^u\mbox{-}\\
			\mbox{and } \mathbb F^0\mbox{-progressively measurable} \big\}.
		\end{aligned}
	\end{equation}
	Here, we note that $\mathbb F^0$ is the filtration generated by the observation process $y^0(\cdot)$ when $u(\cdot) =0$. Specifically, we have

	\begin{definition}\label{Def:Add-Con}
		The set $\mathscr U[0,T]$ defined by \eqref{Add:Con} is called the admissible control set for the players. Each element of $\mathscr U[0,T]$ is referred to as an admissible control of players. 
	\end{definition}
	
We note that the set of admissible controls $\mathscr U[0,T]$ is jointly owned by the two players. The set of admissible controls for each individual player should naturally be taken as the two cross-sections of $\mathscr U[0,T]$. In detail, the set of admissible controls for Player $1$ is defined as
\begin{equation}
\begin{aligned}
\mathscr U_1[u_2(\cdot); 0,T] := \big\{ u_1(\cdot) \in L^2_{\mathbb F^0}(0,T;\mathbb R^{k_1}) \,\big|\\
(u_1(\cdot),u_2(\cdot)) \in \mathscr U[0,T] \big\}
\end{aligned}
\end{equation}
for any given $u_2(\cdot) \in L^2_{\mathbb F^0}(0,T;\mathbb R^{k_2})$ such that there exists a $\bar u_1(\cdot) \in L^2_{\mathbb F^0}(0,T;\mathbb R^{k_1})$ satisfying $(\bar u_1(\cdot), u_2(\cdot)) \in \mathscr U[0,T]$. The set of admissible controls for Player $2$ is defined similarly. However, the measurability of the admissible controls for each individual player depends on the choice of the opponent's admissible controls (see \eqref{Add:Con}). That is, in general, for different $u_{3-i}(\cdot)$ and $u_{3-i}'(\cdot)$,
\[
\mathscr U_i[u_{3-i}(\cdot);0,T] \neq \mathscr U_i[u_{3-i}'(\cdot);0,T],\quad i=1,2.
\]
Hence, unlike the case of complete information, $\mathscr U[0,T]$ cannot be represented as the Cartesian product of the two cross-sections.
	
Now, we use the linearity of the game system \eqref{Sys} and the observation equation \eqref{Obs} to completely transfer the progressive measurability of admissible control to $\mathbb F^0$. In fact, by letting $u(\cdot) \in L^2_{\mathbb F}(0,T;\mathbb R^k)$, the unique solutions $x^u(\cdot)$ and $y^u(\cdot)$ to \eqref{Sys} and \eqref{Obs} admit the following decomposition:
\begin{equation}\label{u=0+1}
x^u(\cdot) = x^0(\cdot) +\xi^u(\cdot),\quad y^u(\cdot) =y^0(\cdot) +\eta^u(\cdot),
\end{equation}
where $x^0(\cdot)$ and $y^0(\cdot)$ are the unique solutions to SDEs \eqref{Sys} and \eqref{Obs} with $u(\cdot) =0$ respectively, $\xi^u(\cdot)$ and $\eta^u(\cdot)$ are the unique solutions to the following ODEs:
	\begin{equation}\label{ODE:x1}
		\left\{
		\begin{aligned}
			& \dot \xi^u(t) =A(t)\xi^u(t) +B(t)u(t),\\
			& \xi^u(0) =0
		\end{aligned}
		\right.
	\end{equation}
	and
	\begin{equation}\label{ODE:y1}
		\left\{
		\begin{aligned}
			& \dot \eta^u(t) = H(t)\xi^u(t),\\
			& \eta^u(0) =0,
		\end{aligned}
		\right.
	\end{equation}
	respectively.

	\begin{lemma}\label{LEM:Fu=F0}
		Let $u(\cdot) \in \mathscr U[0,T]$. Then $\mathbb F^u = \mathbb F^0$.
	\end{lemma}
	
The proof of the lemma can be found in the literature, see Bensoussan \cite[Section 9.1]{B-18} and Wang et al. \cite[Lemma 2.1]{WWX-15} for example.
	


	\begin{definition}\label{Def:law}
		(i) A Borel measurable mapping $\varphi_1:[0,T] \times \mathbb R^n \times \mathbb R^{k_2} \rightarrow \mathbb R^{k_1}$ is called an explicit feedback law for Player $1$ (with partial observation) if the following conditions hold:
		
		(a) For any $u_2(\cdot) \in L^2_{\mathbb F^0}(0,T;\mathbb R^{k_2})$, the CMF-SDE
		\begin{equation}\label{x:varphi1u2}
			\left\{
			\begin{aligned}
				& \mathrm dx^{\varphi_1,u_2}(t) = \Big[ A(t)x^{\varphi_1,u_2}(t) +B_1(t) u_1^{\varphi_1,u_2}(t)\\
				& \qquad\qquad  +B_2(t)u_2(t) +b(t) \Big]\, \mathrm dt\\
				& \qquad\qquad  +C(t)\, \mathrm dW(t) +\bar C(t)\, \mathrm d\bar W(t),\\
				& x^{\varphi_1,u_2}(0) =a,
			\end{aligned}
			\right.
		\end{equation}
with 
\begin{equation}\label{u1:varphi1u2}
u_1^{\varphi_1,u_2}(t) := \varphi_1 \Big( t, \mathbb E\big[ x^{\varphi_1,u_2}(t)\, \big|\, \mathcal F^0_t \big], u_2(t) \Big),
\end{equation}
admits a unique solution $x^{\varphi_1, u_2}(\cdot) \in L^2_{\mathbb F}(\Omega;C(0,T;\mathbb R^n))$. 
		
		(b) The set 
		\begin{equation}
			\begin{aligned}
				\mathscr U_2 [\varphi_1;0,T] := \big\{ u_2(\cdot) \in L^2_{\mathbb F^0}(0,T;\mathbb R^{k_2})\, \big|\\
				(u_1^{\varphi_1,u_2}(\cdot), u_2(\cdot)) \in \mathscr U[0,T] \big\}
			\end{aligned}
		\end{equation}
		is non-empty, where $u_1^{\varphi_1,u_2}(\cdot)$ is given by \eqref{u1:varphi1u2}.
		
		In this case, $u_1^{\varphi_1,u_2}(\cdot)$ is called the induced control of $(\varphi_1, u_2)$ for Player $1$. An explicit feedback law for Player $2$ and the related induced controls can be defined in the same way.

		(ii) A Borel measurable mapping $\psi_2: [0,T] \times \mathbb R^n \rightarrow \mathbb R^{k_2}$ is called an implicit feedback law for Player $2$ (with partial observation) if the following conditions hold:
		
		(a) For any $u_1(\cdot) \in L^2_{\mathbb F^0}(0,T;\mathbb R^{k_1})$, the CMF-SDE
		\begin{equation}
			\left\{
			\begin{aligned}
				& \mathrm dx^{u_1,\psi_2}(t) = \Big[ A(t)x^{u_1,\psi_2}(t) +B_1(t)u_1(t)\\
				& \qquad \qquad +B_2(t)u_2^{u_1,\psi_2}(t) +b(t) \Big]\, \mathrm dt\\
				& \qquad \qquad +C(t)\, \mathrm dW(t) +\bar C(t)\, \mathrm d\bar W(t),\\
				& x^{u_1,\psi_2}(0) =a,
			\end{aligned}
			\right.
		\end{equation}
with
\begin{equation}\label{u2:u1psi2}
u_2^{u_1,\psi_2}(t) := \psi_2 \Big( t, \mathbb E\big[ x^{u_1,\psi_2}(t)\, \big|\, \mathcal F^0_t \big] \Big),
\end{equation}
admits a unique solution $x^{u_1,\psi_2}(\cdot) \in L^2_{\mathbb F}(\Omega;C(0,T;\mathbb R^n))$. 
		
		(b) The set 
		\begin{equation}
			\begin{aligned}
				\mathscr U_1[\psi_2;0,T] := \big\{ u_1(\cdot) \in L^2_{\mathbb F^0}(0,T;\mathbb R^{k_1})\, \big|\\
				(u_1(\cdot), u_2^{u_1,\psi_2}(\cdot)) \in \mathscr U[0,T] \big\}
			\end{aligned}
		\end{equation}
		is non-empty, where $u_2^{u_1,\psi_2}(\cdot)$ is given by \eqref{u2:u1psi2}.
		
		In this case, $u_2^{u_1,\psi_2}(\cdot)$ is called the induced control of $(u_1,\psi_2)$ for Player $2$. An implicit feedback law for Player $1$ and the related induced controls can be defined in the same way.
		
		(iii) Let $\varphi_1$ be an explicit feedback law of Player $1$ and $\psi_2$ an implicit feedback law of Player $2$. The $(\varphi_1,\psi_2)$ is called a pair of explicit-implicit feedback laws (with partial observation) if the following conditions also hold:
		
		(a) The CMF-SDE
		\begin{equation}\label{x:varphi1psi2}
			\left\{
			\begin{aligned}
				& \mathrm dx^{\varphi_1,\psi_2}(t) = \Big[ A(t)x^{\varphi_1,\psi_2}(t) +B_1(t)u_1^{\varphi_1,\psi_2}(t)\\
				& \qquad \qquad +B_2(t)u_2^{\varphi_1,\psi_2}(t) +b(t) \Big]\, \mathrm dt\\
				& \qquad \qquad +C(t)\, \mathrm dW(t) +\bar C(t)\, \mathrm d\bar W(t),\\
				& x^{\varphi_1,\psi_2}(0) =a,
			\end{aligned}
			\right.
		\end{equation}
		with
		\begin{equation}\label{u:varphi1psi2}
			\begin{aligned}
				& u_1^{\varphi_1,\psi_2}(t) := \varphi_1\Big( t,\mathbb E\big[ x^{\varphi_1,\psi_2}(t)\, \big|\, \mathcal F^0_t \big], u_2^{\varphi_1,\psi_2}(t) \Big),\\
				& u_2^{\varphi_1,\psi_2}(t) := \psi_2\Big( t,\mathbb E\big[ x^{\varphi_1,\psi_2}(t)\, \big|\, \mathcal F^0_t \big] \Big),
			\end{aligned}
		\end{equation}
		admits a unique solution $x^{\varphi_1,\psi_2}(\cdot) \in L^2_{\mathbb F}(\Omega;C(0,T;\mathbb R^n))$.
		
		(b) $u^{\varphi_1,\psi_2}(\cdot):=(u_1^{\varphi_1,\psi_2}(\cdot), u_2^{\varphi_1,\psi_2}(\cdot))\in\mathscr U[0,T]$.
		
		In this case, $u^{\varphi_1,\psi_2}(\cdot)$ given by \eqref{u:varphi1psi2} is called the induced control of $(\varphi_1,\psi_2)$. A pair of implicit-explicit feedback laws and the related induced control can be defined in the same way.
	\end{definition}

In Definition \ref{Def:law}(i), the modifier ``explicit'' means that the mapping $\varphi_1$ depends on the variable $u_2$, thereby the dependence of the induced control $u_1^{\varphi_1,u_2}(\cdot)$ of Player $1$ on the control $u_2(\cdot)$ of Player $2$ is explicitly revealed by the defining equation \eqref{u1:varphi1u2}. In contrast, the modifier ``implicit'' in Definition \ref{Def:law}(ii) indicates that the mapping $\psi_2$ does not depend on the variable $u_1$, thereby the dependence of the induced control $u_2^{u_1,\psi_2}(\cdot)$ of Player $2$ on the control $u_1(\cdot)$ of Player $1$ is implicitly expressed by the defining equation \eqref{u2:u1psi2}, that is, $u_2^{u_1,\psi_2}(\cdot)$ implicitly depends on $u_1(\cdot)$ through the conditional expectation of the state $x^{u_1,\psi_2}(\cdot)$.

\begin{definition}\label{Def:Sad-P}
A pair of explicit-implicit feedback laws $(\varphi_1, \psi_2)$ is called a saddle point if
\begin{equation}\label{SP:EI}
\begin{aligned}
J\big( u^{\varphi_1,\psi_2}(\cdot) \big) 
=\ & \inf_{u_2(\cdot) \in \mathscr U_2 [\varphi_1;0,T]} J \big( u_1^{\varphi_1,u_2}(\cdot), u_2(\cdot) \big)\\
=\ & \sup_{u_1(\cdot) \in \mathscr U_1 [\psi_2;0,T]} J\big( u_1(\cdot), u_2^{u_1,\psi_2}(\cdot)\big).
\end{aligned}
\end{equation}
A saddle point in implicit-explicit feedback law form can be defined similarly. 
\end{definition}

These saddle points characterize a type of equilibrium: when the opponent plays according to her/his feedback law in the saddle point, the player's best choice is also to play according to her/his own feedback law in the saddle point; otherwise, the player will be punished, that is, her/his own ``payoff decreases'' or ``cost increases''. In practice, this may be interpreted as: both players in the game must follow their respective rules. If either player violates the rules, she/he will be punished. This ensures that both players act according to the rules, thus maintaining the entire dynamic game in a state of equilibrium.

	{\bf Problem (LQG)} Find saddle points in explicit-implicit/implicit-explicit feedback law form for the LQ game with partial observation.

	\section{Filtering and separation}\label{Sec:F-S}
	
	Inspired by Lemma \ref{LEM:Fu=F0}, we let $u(\cdot) \in L^2_{\mathbb F^0}(0,T;\mathbb R^k)$ in this section. For any $t\in [0,T]$, we introduce the following notations:
	\begin{equation}
		\widehat x^u(t) := \mathbb E\big[ x^u(t)\, \big|\, \mathcal F^0_t \big]  \ \ \mbox{and} \ \ 
		\widetilde x^u(t) := x^u(t) -\widehat x^u(t)
	\end{equation}
	which are respectively called the filtering estimate of $x^u(t)$ with respect to $\mathcal F^0_t$ and the related estimation error. Clearly, due to the decomposition \eqref{u=0+1}, we have
	\begin{equation}\label{hat:u=0+1}
		\widehat x^u(t) = \mathbb E\big[ x^0(t)\, \big|\, \mathcal F^0_t \big] +\xi^u(t) = \widehat x^0(t) +\xi^u(t)
	\end{equation}
	and
	\begin{equation}\label{tilde xu=x0}
		\widetilde x^u(t) = \big[ x^0(t) +\xi^u(t) \big] -\big[ \widehat x^0(t) +\xi^u(t) \big] = \widetilde x^0(t).
	\end{equation}
	We note that the last equation indicates that $\widetilde x^u(\cdot)$ is independent of the players' controls. Therefore, we drop the superscript and rewrite $\widetilde x^u(\cdot)$ as $\widetilde x(\cdot)$.
	
	
We introduce the following
	
	\noindent{\bf Assumption (H1)} The matrix-valued process $K(\cdot)^{-1}$ exists and is uniformly bounded on $[0,T]$.
	
Similar to Bensoussan \cite[Proposition 9.1]{B-18}, we have the following

\begin{proposition}\label{Prop:filter}
Let Assumption (H1) hold and $u(\cdot) \in L^2_{\mathbb F^0}(0,T;\mathbb R^k)$. Then, the filtering estimate $\widehat x^u(\cdot)$ satisfies the following SDE:
\begin{equation}\label{F:Sys}
\left\{
\begin{aligned}
& \mathrm d\widehat x^u(t) = \big[ A(t)\widehat x^u(t) +B(t)u(t) +b(t) \big]\, \mathrm dt\\
& \qquad \qquad +D(t)\, \mathrm d\widehat W(t),\\
& \widehat x^u(0) =a,
\end{aligned}
\right.
\end{equation}
where 
\begin{equation}\label{D}
D(\cdot) := C(\cdot) +\Sigma(\cdot)\big[ K(\cdot)^{-1}H(\cdot) \big]^\top,
\end{equation}
$\Sigma(\cdot) := \mathbb E[ \widetilde x(\cdot) \widetilde x(\cdot)^\top]$ is the covariance matrix process of $\widetilde x(\cdot)$ and can be represented as the unique solution to the following Riccati equation (the argument $t$ is suppressed):
\begin{equation}\label{KB:Riccati}
\left\{
\begin{aligned}
& \dot \Sigma = \big[ A -CK^{-1}H \big] \Sigma +\Sigma \big[ A -CK^{-1}H \big]^\top\\
& \qquad -\Sigma H^\top \big[ KK^\top \big]^{-1} H\Sigma +\bar C\bar C^\top,\\
& \Sigma(0) =0,
\end{aligned}
\right.
\end{equation}
and $\widehat W(\cdot)$ is an $m$-dimensional standard Brownian motion which is intrinsically independent of the choice of $u(\cdot)$, yet is defined via $u(\cdot)$ as follows:
\begin{equation}\label{hatW:u}
\begin{aligned}
\widehat W(\cdot) :=\ & \int_0^\cdot K(s)^{-1} \Big\{ \mathrm dy^u(s) \\
& -[ H(s) \widehat x^u(s) +h(s) ]\, \mathrm ds \Big\}.
\end{aligned}
\end{equation}
Moreover, $\mathbb F^{\widehat W} = \mathbb F^0$.
\end{proposition}


	Now we turn our attention to the objective functional defined by \eqref{J} to obtain its decomposition corresponding to $x^u(\cdot) = \widehat x^u(\cdot) +\widetilde x(\cdot)$ as follows:
	\begin{equation}\label{WonhamSP}
		J(u(\cdot)) = \widehat J(u(\cdot)) +\widetilde J,
	\end{equation}
	where
	\begin{equation}\label{J:hat}
		\begin{aligned}
			& \widehat J(u(\cdot)) := \frac 1 2 \mathbb E \Big\{ \big\langle G\widehat x^u(T),\, \widehat x^u(T) \big\rangle +2\langle g,\, \widehat x^u(T) \rangle\\
			& \ \ +\int_0^T \Bigg[ \left\langle \begin{pmatrix} Q(t) & S(t)^\top \\ S(t) & R(t) \end{pmatrix} \begin{pmatrix} \widehat x^u(t) \\ u(t) \end{pmatrix},\, \begin{pmatrix} \widehat x^u(t) \\ u(t) \end{pmatrix} \right\rangle\\
			& \ \ +2 \left\langle \begin{pmatrix} q(t) \\ \rho(t) \end{pmatrix},\, \begin{pmatrix} \widehat x^u(t) \\ u(t) \end{pmatrix} \right\rangle \Bigg]\, \mathrm dt \Bigg\}
		\end{aligned}
	\end{equation}
	and
	\begin{equation}\label{tildeJ}
		\begin{aligned}
			& \widetilde J\\
			:=\ & \frac 1 2 \mathbb E \bigg\{ \langle G \widetilde x(T), \widetilde x(T) \rangle +\int_0^T \langle Q(t)\widetilde x(t), \widetilde x(t) \rangle\, \mathrm dt \bigg\}\\
			=\ & \frac 1 2 \bigg\{ \mbox{tr}\, \big( G\Sigma(T) \big) +\int_0^T \mbox{tr}\, \big( Q(t)\Sigma(t) \big)\, \mathrm dt \bigg\}
		\end{aligned}
	\end{equation}
(the last equation is due to $\Sigma(\cdot) := \mathbb E[ \widetilde x(\cdot) \widetilde x(\cdot)^\top]$, see Proposition \ref{Prop:filter}), for any $u(\cdot) \in L^2_{\mathbb F^0}(0,T;\mathbb R^k)$.

\section{Completion of squares and saddle points}\label{Sec:Com-Sad}
	
Inspired by the work of Yu \cite{Y-15} on an LQ zero-sum stochastic differential game with complete information, we employ the method of completion of squares to obtain the desired feedback laws.

First of all, we introduce the following

\noindent{\bf Assumption (H2)} The matrix-valued process $R(\cdot)^{-1}$ exists and is uniformly bounded on $[0,T]$.
	
Under Assumption (H2), another Riccati equation is also introduced (it differs from \eqref{KB:Riccati} and the argument $t$ is suppressed for simplicity):
	\begin{equation}\label{Riccati}
		\left\{
		\begin{aligned}
			& -\dot P = P A +A^\top P +Q -\Theta^\top R^{-1}\Theta,\\
			& P(T) =G,
		\end{aligned}
		\right.
	\end{equation}
where 
	\begin{equation}\label{Theta}
		\Theta := B^\top P +S = \begin{pmatrix} B_1^\top P +S_1 \\
			B_2^\top P +S_2 \end{pmatrix} =: \begin{pmatrix} \Theta_1 \\ \Theta_2 \end{pmatrix}.
	\end{equation}
	
	We note that since the diffusion of the filtering equation \eqref{F:Sys} does not depend on the state and control, \eqref{Riccati} can be regarded as a Riccati equation corresponding to a deterministic LQ problem. However, unlike the optimal control problem, the solvability of Riccati equation \eqref{Riccati} associated with the game problem is quite challenging in the general case. In this section, in order to focus on the main aim of feedback laws, we would like to assume the following
	
	\noindent{\bf Assumption (H3)} The Riccati equation \eqref{Riccati} admits a solution $P(\cdot) \in L^\infty(0,T;\mathbb S^n)$.

	In fact, Tian et al. \cite[Proposition A.3]{TYZh-20} and Sun \cite[Theorem 4.3]{S-21} provided two unique solvability results for \eqref{Riccati} in special cases.

We note that the filtering equation \eqref{F:Sys} contains non-homogeneous terms and the functional $\widehat J$ (see \eqref{J:hat}) contains first-order terms. Therefore, unlike \cite{Y-15}, we also need to introduce an additional ODE ($t$ is suppressed):
	\begin{equation}\label{ODE:p}
		\left\{
		\begin{aligned}
			& -\dot p = A^\top p +Pb +q -\Theta^\top R^{-1} \nu,\\
			& p(T) =g,
		\end{aligned}
		\right.
	\end{equation}
	where
	\begin{equation}\label{nu}
		\nu := B^\top p +\rho = \begin{pmatrix} B_1^\top p +\rho_1 \\ B_2^\top p +\rho_2 \end{pmatrix} =: \begin{pmatrix} \nu_1 \\ \nu_2 \end{pmatrix}.
	\end{equation}
	It is clear that, when $P(\cdot) \in L^\infty(0,T;\mathbb S^n)$, ODE \eqref{ODE:p} admits a unique solution.

We recall that when Yu \cite{Y-15} studied a zero-sum LQ game problem, a two-step complication-of-squares procedure was introduced, i.e., the square was first completed with respect to $u_1(\cdot)$ (resp. $u_2(\cdot)$) and then with respect to $u_2(\cdot)$ (resp. $u_1(\cdot)$). Instead, the remainder of this section consolidates the two-step procedure of \cite{Y-15} into the classical single-step treatment of LQ optimal control, which completes the square jointly for $u(\cdot) =(u_1(\cdot),u_2(\cdot))$, and incorporates the block-matrix diagonalization technique (see \eqref{E1Lambda1} and \eqref{E2Lambda2}).

Let $u(\cdot)\in L^2_{\mathbb F^0}(0,T;\mathbb R^k)$. We apply It\^o's formula to $\langle P(\cdot)\widehat x^u(\cdot),\, \widehat x^u(\cdot) \rangle +2\langle p(\cdot),\, \widehat x^u(\cdot) \rangle$ to have
	\[
	\begin{aligned}
		& \mathbb E\Big[ \big\langle G\widehat x^u(T),\, \widehat x^u(T) \big\rangle +2\big\langle g,\, \widehat x^u(T) \big\rangle \Big]\\
		=\ & \langle P(0)a,\, a \rangle +2\langle p(0),\, a \rangle +\mathbb E \int_0^T \Big\{ 2\langle p,\, b \rangle\\
		& +\mbox{tr}\, \big( DD^\top P \big) -\big\langle Q\widehat x^u,\, \widehat x^u \big\rangle -2 \langle q,\, \widehat x^u \rangle\\
		& +\big\langle R^{-1}\Theta \widehat x^u,\, \Theta\widehat x^u +2\nu \big\rangle\\
		& +2 \big\langle u,\, B^\top P \widehat x^u +B^\top p \big\rangle \Big\}\, \mathrm dt.
	\end{aligned}
	\]
Then, from
	\[
	\big\langle R^{-1}\Theta \widehat x^u,\, \Theta\widehat x^u +2\nu \big\rangle =\big\langle R^{-1}\Delta^u,\, \Delta^u \big\rangle -\langle R^{-1}\nu,\, \nu \rangle
	\]
	with
	\begin{equation}\label{Delta}
		\Delta^u := \Theta\widehat x^u +\nu =\begin{pmatrix} \Theta_1\widehat x^u +\nu_1 \\ \Theta_2\widehat x^u +\nu_2 \end{pmatrix} =: \begin{pmatrix} \Delta^u_1 \\ \Delta^u_2 \end{pmatrix},
	\end{equation}
we immediately get
\[
	\begin{aligned}
		& \mathbb E\Big[ \big\langle G\widehat x^u(T),\, \widehat x^u(T) \big\rangle +2\big\langle g,\, \widehat x^u(T) \big\rangle \Big]\\
		=\ & \Gamma +\mathbb E \int_0^T \Big\{ -\big\langle Q\widehat x^u,\, \widehat x^u \big\rangle -2 \langle q,\, \widehat x^u \rangle\\
		& +\big\langle R^{-1}\Delta^u,\, \Delta^u \big\rangle +2 \big\langle u,\, B^\top P \widehat x^u +B^\top p \big\rangle \Big\}\, \mathrm dt,
	\end{aligned}
	\]
	where the constant $\Gamma$ is defined by
	\begin{equation}\label{Gamma}
		\begin{aligned}
			\Gamma :=\ & \langle P(0)a,\, a \rangle +2\langle p(0),\, a \rangle +\int_0^T \Big\{ 2\langle p,\, b \rangle\\
			& +\mbox{tr}\, \big( DD^\top P \big) -\langle R^{-1}\nu,\, \nu \rangle \Big\}\, \mathrm dt.
		\end{aligned}
	\end{equation}
	By the definition \eqref{J:hat} of $\widehat J$ and the completion of square, we further obtain
\begin{equation}\label{CS0}
		\begin{aligned}
			& 2\widehat J(u(\cdot)) - \Gamma\\
			=\ & \mathbb E\int_0^T \Big\{ \langle Ru, u \rangle +2 \langle u, \Delta^u \rangle +\langle R^{-1}\Delta^u, \Delta^u \rangle \Big\}\, \mathrm dt\\
			=\ & \mathbb E\int_0^T \Big\langle R\big[ u+R^{-1}\Delta^u \big],\, u+R^{-1}\Delta^u \Big\rangle \, \mathrm dt
		\end{aligned}
	\end{equation}
	with the notations $\Gamma$, $\Delta^u$, $\Theta$ and $\nu$ defined by \eqref{Gamma}, \eqref{Delta}, \eqref{Theta} and \eqref{nu}, respectively.


\subsection{A saddle point in explicit-implicit feedback law form}
	
As usual, for a mapping $M:[0,T] \rightarrow \mathbb S^n$, if there exists a constant $\delta>0$ such that $M(t)-\delta I_n$ is positive semi-definite (resp. $M(t)+\delta I_n$ is negative semi-definite) for almost all $t\in [0,T]$, we call $M(\cdot)$ uniformly positive definite (resp. uniformly negative definite), denoted by $M(\cdot) \gg 0$ (resp. $M(\cdot) \ll 0$).

For our game problem, we propose the following 
%

{\bf Condition (I)} $R_{11}(\cdot) \ll 0$ and $[R_{22} -R_{21}R_{11}^{-1} R_{21}^\top](\cdot) \gg 0$.

Under Condition (I), the matrix $R$ can be block-diagonalized as follows (the argument $t$ is suppressed):
\begin{equation}\label{E1Lambda1}
R = E_1^{-1} \Lambda_1 (E_1^\top)^{-1},
\end{equation}
where
\begin{equation}\label{E1}
E_1:= \begin{pmatrix} I & 0 \\ -R_{21}R_{11}^{-1} & I \end{pmatrix}
\end{equation}
and
\begin{equation}\label{Lambda1}
\Lambda_1 := \begin{pmatrix} R_{11} & 0 \\ 0 & R_{22} -R_{21}R_{11}^{-1}R_{21}^\top \end{pmatrix}.
\end{equation}
From the above, it is obvious that Condition (I) implies Assumption (H2). Moreover,
%
%
	\begin{equation}\label{EI:AE}
		\begin{aligned}
			u+R^{-1}\Delta^u =\ & E_1^\top \big[ (E_1^\top)^{-1}u +\Lambda_1^{-1} E_1 \Delta^u \big]\\
			=\ & E_1^\top \begin{pmatrix} u_1 -\varphi_1(\widehat x^u, u_2)\\ u_2 -\psi_2(\widehat x^u) \end{pmatrix},
		\end{aligned}
	\end{equation}
	where the mappings $\varphi_1:[0,T]\times \mathbb R^n \times \mathbb R^{k_2} \rightarrow \mathbb R^{k_1}$ and $\psi_2: [0,T] \times \mathbb R^n \rightarrow \mathbb R^{k_2}$ are defined by
	\begin{equation}\label{EI:varphi1}
		\begin{aligned}
			\varphi_1(t,x,u_2) := -R_{11}(t)^{-1} \big[ R_{21}(t)^\top u_2\\
			+\Theta_1(t)x +\nu_1(t) \big]
		\end{aligned}
	\end{equation}
	and
	\begin{equation}\label{EI:psi2}
		\begin{aligned}
			\psi_2(t,x) := -\big[ R_{22} -R_{21}R_{11}^{-1}R_{21}^\top \big]^{-1}(t) \Big\{ \Theta_2(t)x\\
			+\nu_2(t) -\big[ R_{21}R_{11}^{-1} \big](t) \big[ \Theta_1(t)x +\nu_1(t) \big] \Big\}
		\end{aligned}
	\end{equation}
	with the notations $\Theta_1(t)$, $\Theta_2(t)$, $\nu_1(t)$ and $\nu_2(t)$ defined by \eqref{Theta} and \eqref{nu}, for any $(t,x) \in [0,T]\times \mathbb R^n$ and any $u_2 \in \mathbb R^{k_2}$, respectively.

	With the help of \eqref{E1Lambda1}, \eqref{EI:AE} and \eqref{Lambda1}, Equation \eqref{CS0} is reduced to
	\begin{equation}\label{CS1}
		\begin{aligned}
			& 2\widehat J(u(\cdot)) = \Gamma + \mathbb E\int_0^T \Big\{ \Big\langle R_{11} \big[ u_1 -\varphi_1(\widehat x^u, u_2) \big],\\
			& \qquad\quad u_1 -\varphi_1(\widehat x^u, u_2) \Big\rangle + \Big\langle \big[ R_{22}-R_{21}R_{11}^{-1}R_{21}^\top \big]\\
			& \qquad\quad \times \big[ u_2 -\psi_2(\widehat x^u) \big],\, u_2 -\psi_2(\widehat x^u) \Big\rangle\Big\}\, \mathrm dt.
		\end{aligned}
	\end{equation}

	\begin{lemma}\label{Lem:nonempty}
		Let Assumptions (H1) and (H3) and Condition (I) hold. Let the pair of mappings $(\varphi_1,\psi_2)$ be defined by \eqref{EI:varphi1} and \eqref{EI:psi2}. Then,
		\begin{enumerate}[(a)]
			\item the CMF-SDE \eqref{x:varphi1psi2} admits a unique solution $x^{\varphi_1,\psi_2}(\cdot) \in L^2_{\mathbb F}(\Omega;C(0,T;\mathbb R^n))$;
			
			\item the $u^{\varphi_1,\psi_2}(\cdot)$ defined by \eqref{u:varphi1psi2} belongs to $\mathscr U[0,T]$.
		\end{enumerate}
	\end{lemma}
	
	\begin{proof}
		(a) Due to the specific definitions of $\varphi_1$ and $\psi_2$ (see \eqref{EI:varphi1} and \eqref{EI:psi2}), the $u^{\varphi_1,\psi_2}(\cdot)$ involved in CMF-SDE \eqref{x:varphi1psi2} is correspondingly reduced to
		\begin{equation}\label{EI:u}
			\begin{aligned}
				& u_1^{\varphi_1,\psi_2}(\cdot) = -R_{11}(\cdot)^{-1}\Big[ R_{21}(\cdot)^\top u_2^{\varphi_1,\psi_2}(\cdot)\\
				& \qquad +\Theta_1(\cdot) \widehat x^{\varphi_1,\psi_2}(\cdot) +\nu_1(\cdot) \Big],\\
				& u_2^{\varphi_1,\psi_2}(\cdot) = -\big[ R_{22} -R_{21}R_{11}^{-1}R_{21}^\top \big]^{-1}(\cdot)\\
				& \qquad \times \Big\{ \Theta_2(\cdot) \widehat x^{\varphi_1,\psi_2}(\cdot) +\nu_2(\cdot)\\
				& \qquad -\big[ R_{21}R_{11}^{-1} \big](\cdot) \big[ \Theta_1(\cdot) \widehat x^{\varphi_1,\psi_2}(\cdot) +\nu_1(\cdot) \big] \Big\}.
			\end{aligned}
		\end{equation}
		Here we recall the notation $\widehat x^{\varphi_1,\psi_2}(\cdot) := \mathbb E[x^{\varphi_1,\psi_2}(\cdot)\, |\, \mathcal F^0_\cdot]$. Due to the boundedness of the involved matrix-valued processes in \eqref{EI:u}, Proposition 7.2 in Carmona and Zhu \cite{CZh-16} implies the unique solvability of CMF-SDE \eqref{x:varphi1psi2} in the space $L^2_{\mathbb F}(\Omega;C(0,T;\mathbb R^n))$.
		
(b) From Conclusion (a), we know that $\widehat x^{\varphi_1,\psi_2}(\cdot)$ is square integrable and $\mathbb F^0$-progressively measurable. Then, the expression \eqref{EI:u} implies that $u^{\varphi_1,\psi_2}(\cdot)$ also satisfies both square integrability and $\mathbb F^0$-progressive measurability. Recalling the definition \eqref{Add:Con} of $\mathscr U[0,T]$, we still need to verify that $u^{\varphi_1,\psi_2}(\cdot)$ is $\mathbb F^{u^{\varphi_1,\psi_2}}$-progressively measurable. Again, in view of \eqref{EI:u}, we can achieve this by establishing the $\mathbb F^{u^{\varphi_1,\psi_2}}$-progressive measurability of $\widehat x^{\varphi_1,\psi_2}(\cdot)$.

		
In fact, according to Proposition \ref{Prop:filter}, especially the representation \eqref{hatW:u} of $\widehat W(\cdot)$, the filtering equation related to $x^{\varphi_1,\psi_2}(\cdot)$ can be rewritten as (the
argument $t$ is suppressed for simplicity)
\begin{equation}\label{EI:x:varphi1psi2}
\left\{
\begin{aligned}
& \mathrm d\widehat x^{\varphi_1,\psi_2} = \big[ \mathbf A^{\varphi_1,\psi_2} \widehat x^{\varphi_1,\psi_2} +\mathbf b^{\varphi_1,\psi_2} \big]\, \mathrm dt\\
& \qquad +DK^{-1}\Big\{ \mathrm dy^{u^{\varphi_1,\psi_2}} -\big[ H\widehat x^{\varphi_1,\psi_2} +h \big]\, \mathrm dt \Big\},\\
& \widehat x^{\varphi_1,\psi_2}(0) =a,
\end{aligned}
\right.
\end{equation}
where
\begin{equation}
\begin{aligned}
& \mathbf A^{\varphi_1,\psi_2} := [A -B_1R_{11}^{-1}\Theta_1] -[B_2 -B_1R_{11}^{-1}R_{21}^\top]\\
& \qquad \times [R_{22} -R_{21}R_{11}^{-1}R_{21}^\top]^{-1} [\Theta_2 -R_{21}R_{11}^{-1}\Theta_1],\\
& \mathbf b^{\varphi_1,\psi_2} := [b -B_1R_{11}^{-1}\nu_1] -[B_2 -B_1R_{11}^{-1}R_{21}^\top]\\
& \qquad \times [R_{22} -R_{21}R_{11}^{-1}R_{21}^\top]^{-1} [\nu_2 -R_{21}R_{11}^{-1}\nu_1].
\end{aligned}
\end{equation}
Since the randomness of SDE \eqref{EI:x:varphi1psi2} only originates from $y^{u^{\varphi_1,\psi_2}}(\cdot)$, its solution $\widehat x^{\varphi_1,\psi_2}(\cdot)$ is $\mathbb F^{u^{\varphi_1,\psi_2}}$-progressively measurable. We finish the proof.
	\end{proof}

	\begin{theorem}\label{THM:EI-Sad}
		Let Assumptions (H1) and (H3) and Condition (I) hold. Then, $(\varphi_1, \psi_2)$ defined by \eqref{EI:varphi1} and \eqref{EI:psi2} is a saddle point in explicit-implicit feedback law form. Moreover, 
		\begin{equation}\label{EI:value}
			J\big( u^{\varphi_1,\psi_2}(\cdot) \big) = \frac {\Gamma}{2} +\widetilde J,
		\end{equation}
		where the constants $\Gamma$ and $\widetilde J$ are given by \eqref{Gamma} and \eqref{tildeJ}, respectively.
	\end{theorem}
	
	\begin{proof}
		The whole proof is divided into four steps.
		
		(i) {\it $\varphi_1$ defined by \eqref{EI:varphi1} is an explicit feedback law.}
		
		Firstly, for any $u_2(\cdot) \in L^2_{\mathbb F^0}(0,T;\mathbb R^{k_2})$, due to the specific definition of $\varphi_1$ (see \eqref{EI:varphi1}), the $u_1^{\varphi_1,u_2}(\cdot)$ involved in CMF-SDE \eqref{x:varphi1u2} is correspondingly reduced to
		\begin{equation}\label{EI:u1}
			\begin{aligned}
				u_1^{\varphi_1,u_2}(\cdot) =\ & -R_{11}(\cdot)^{-1} \Big[ R_{21}(\cdot)^\top u_2(\cdot)\\
				& +\Theta_1(\cdot) \widehat x^{\varphi_1,u_2}(\cdot) +\nu_1(\cdot) \Big].
			\end{aligned}
		\end{equation}
		Similar to the proof of Lemma \ref{Lem:nonempty}(a), the boundedness of the involved matrix-valued processes in \eqref{EI:u1} and Proposition 7.2 in \cite{CZh-16} work once again to ensure that CMF-SDE \eqref{x:varphi1u2} admits a unique solution $x^{\varphi_1,u_2}(\cdot) \in L^2_{\mathbb F}(\Omega; C(0,T;\mathbb R^n))$.
		
		Secondly, by a direct verification, it can be known that $u_2^{\varphi_1,\psi_2}(\cdot)$ defined by \eqref{EI:u} belongs to $\mathscr U_2[\varphi_1;0,T]$.
		
		We have validated Definition \ref{Def:law}(i), and completed the proof of this step.
		
		(ii) {\it $\psi_2$ defined by \eqref{EI:psi2} is an implicit feedback law.}
		
		The proof is similar to that of Step (i). We omit it.
		
		(iii) {\it $(\varphi_1,\psi_2)$ is a pair of explicit-implicit feedback laws.}
		
		Combining Step (i), Step (ii) and Lemma \ref{Lem:nonempty}, we get this conclusion.
		
		(iv) {\it $(\varphi_1,\psi_2)$ is a saddle point.}
		
		Let us go back to \eqref{CS1}. Due to Condition (I), for any $u_1(\cdot) \in \mathscr U_1[\psi_2;0,T]$ and any $u_2(\cdot) \in \mathscr U_2[\varphi_1;0,T]$, we have
		\[
		\begin{aligned}
			& \widehat J\big( u_1^{\varphi_1,u_2}(\cdot), u_2(\cdot) \big) = \frac{\Gamma}{2} +\frac 1 2 \mathbb E\int_0^T \Big\langle \big[ R_{22}-R_{21}R_{11}^{-1}R_{21}^\top \big]\\
			& \qquad \times \big[ u_2 -\psi_2(\widehat x^{\varphi_1,u_2}) \big],\, u_2 -\psi_2(\widehat x^{\varphi_1,u_2}) \Big\rangle\, \mathrm dt\\
			& \geq \frac{\Gamma}{2} = \widehat J\big( u^{\varphi_1,\psi_2}(\cdot) \big)\\
			& \geq \frac{\Gamma}{2} + \frac 1 2 \mathbb E\int_0^T \Big\langle R_{11} \big[ u_1-\varphi_1\big( \widehat x^{u_1,\psi_2}, \psi_2(\widehat x^{u_1,\psi_2}) \big) \big],\, u_1\\
			& \qquad -\varphi_1\big( \widehat x^{u_1,\psi_2}, \psi_2(\widehat x^{u_1,\psi_2}) \big) \Big\rangle\, \mathrm dt = \widehat J \big( u_1(\cdot), u_2^{u_1,\psi_2}(\cdot) \big).
		\end{aligned}
		\]
		Moreover, with the help of \eqref{WonhamSP}, we obtain \eqref{SP:EI} from above, i.e. $(\varphi_1,\psi_2)$ is a saddle point. At the same time, \eqref{WonhamSP} also implies \eqref{EI:value}. The proof is completed.
\end{proof}

\subsection{A saddle point in implicit-explicit feedback law form}
	
The content of this subsection is symmetric to that of the previous subsection, then we state the conclusions succinctly. 

We propose

{\bf Condition (II)} $R_{22}(\cdot) \gg 0$ and $[R_{11}-R_{21}^\top R_{22}^{-1}R_{21}](\cdot) \ll 0$.
	

Under Condition (II), the matrix $R$ is block-diagonalized as follows ($t$ is suppressed):
\begin{equation}\label{E2Lambda2}
R = E_2^{-1} \Lambda_2 (E_2^\top)^{-1},
\end{equation}
where
\begin{equation}\label{E2}
E_2 := \begin{pmatrix} I & -R_{21}^\top R_{22}^{-1} \\ 0 & I \end{pmatrix}
\end{equation}
and
\begin{equation}\label{Lambda2}
\Lambda_2 := \begin{pmatrix} R_{11} -R_{21}^\top R_{22}^{-1}R_{21} & 0 \\ 0 & R_{22} \end{pmatrix}.
\end{equation}
Clearly, Condition (II) also implies Assumption (H2). Furthermore, in this case,
%
\begin{equation}\label{IE:AE}
		u+R^{-1}\Delta^u = E_2^\top \begin{pmatrix} u_1 -\psi_1(\widehat x^u)\\ u_2 -\varphi_2(\widehat x^u, u_1) \end{pmatrix},
	\end{equation}
	where the mappings $\psi_1:[0,T] \times \mathbb R^n \rightarrow \mathbb R^{k_1}$ and $\varphi_2:[0,T] \times \mathbb R^n \times \mathbb R^{k_1} \rightarrow \mathbb R^{k_2}$ are defined by
	\begin{equation}\label{IE:psi1}
		\begin{aligned}
			\psi_1(t,x) := -\big[ R_{11} -R_{21}^\top R_{22}^{-1} R_{21} \big]^{-1}(t) \Big\{ \Theta_1(t)x\\
			+\nu_1(t) -\big[ R_{21}^\top R_{22}^{-1} \big](t) \big[ \Theta_2(t)x +\nu_2(t) \big] \Big\}
		\end{aligned}
	\end{equation}
	and
	\begin{equation}\label{IE:varphi2}
		\begin{aligned}
			\varphi_2(t,x,u_1) := -R_{22}(t)^{-1} \big[ R_{21}(t)u_1\\
			+\Theta_2(t)x +\nu_2(t) \big],
		\end{aligned}
	\end{equation}
for any $(t,x) \in [0,T]\times \mathbb R^n$ and any $u_1 \in \mathbb R^{k_1}$, respectively.


	With the help of \eqref{E2Lambda2}, \eqref{IE:AE} and \eqref{Lambda2}, Equation \eqref{CS0} can be reduced to
	\begin{equation}\label{CS2}
		\begin{aligned}
			& 2\widehat J(u(\cdot)) = \Gamma + \mathbb E\int_0^T \Big\{ \Big\langle \big[ R_{11}-R_{21}^\top R_{22}^{-1} R_{21} \big]\\
			& \qquad \times \big[ u_1 -\psi_1(\widehat x^u) \big],\, u_1 -\psi_1(\widehat x^u) \Big\rangle + \Big\langle R_{22}\\
			& \qquad \times \big[ u_2-\varphi_2(\widehat x^u, u_1) \big],\, u_2-\varphi_2(\widehat x^u, u_1) \Big\rangle \Big\}\, \mathrm dt.
		\end{aligned}
	\end{equation}

Similar to Lemma \ref{Lem:nonempty} and Theorem \ref{THM:EI-Sad}, we have the following results.
	
	\begin{lemma}
		Let Assumptions (H1) and (H3) and Condition (II) hold. Let the pair of mappings $(\psi_1,\varphi_2)$ be defined by \eqref{IE:psi1} and \eqref{IE:varphi2}. Then, the following results hold:
		
		(a) The CMF-SDE
		\begin{equation}
			\left\{
			\begin{aligned}
				& \mathrm dx^{\psi_1,\varphi_2}(t) = \Big[ A(t)x^{\psi_1,\varphi_2}(t) +B_1(t)u_1^{\psi_1,\varphi_2}(t)\\
				& \qquad \qquad +B_2(t)u_2^{\psi_1,\varphi_2}(t) +b(t) \Big]\, \mathrm dt\\
				& \qquad \qquad +C(t)\, \mathrm dW(t) +\bar C(t)\, \mathrm d\bar W(t),\\
				& x^{\psi_1,\varphi_2}(0) =a,
			\end{aligned}
			\right.
		\end{equation}
		with
		\begin{equation}\label{u:psi1varphi2}
			\begin{aligned}
				& u_1^{\psi_1,\varphi_2}(t) := \psi_1\Big( t,\mathbb E\big[ x^{\psi_1,\varphi_2}(t)\, \big|\, \mathcal F^0_t \big] \Big),\\
				& u_2^{\psi_1,\varphi_2}(t) := \varphi_2\Big( t,\mathbb E\big[ x^{\psi_1,\varphi_2}(t)\, \big|\, \mathcal F^0_t \big], u_1^{\psi_1,\varphi_2}(t) \Big),
			\end{aligned}
		\end{equation}
		admits a unique solution $x^{\psi_1,\varphi_2}(\cdot) \in L^2_{\mathbb F}(\Omega;C(0,T;\mathbb R^n))$.
		
		(b) The $u^{\psi_1,\varphi_2}(\cdot)$ defined by \eqref{u:psi1varphi2} belongs to $\mathscr U[0,T]$.
	\end{lemma}
	

	
	\begin{theorem}\label{THM:IE-Sad}
		Let Assumptions (H1) and (H3) and Condition (II) hold. Then, $(\psi_1,\varphi_2)$ defined by \eqref{IE:psi1} and \eqref{IE:varphi2} is a saddle point in implicit-explicit feedback law form. Moreover, 
\begin{equation}
J\big( u^{\psi_1,\varphi_2}(\cdot) \big) = \frac {\Gamma}{2} +\widetilde J,
\end{equation}
where $\Gamma$ and $\widetilde J$ are given by \eqref{Gamma} and \eqref{tildeJ}, respectively.
	\end{theorem}

	\subsection{The case where Conditions (I) and (II) hold}
	
	At the end of this section, we would like to pay attention to the special case where both Condition (I) and Condition (II) hold true simultaneously. Obviously, these two conditions are equivalent to
	
	{\bf Condition (I \& II)} $R_{11}(\cdot) \ll 0$ and $R_{22}(\cdot) \gg 0$.


	\begin{corollary}\label{Cor:Same-Sad}
Let Assumptions (H1) and (H3) and Condition (I \& II) hold. Then, $(\varphi_1, \psi_2)$ defined by \eqref{EI:varphi1} and \eqref{EI:psi2} is a saddle point in explicit-implicit feedback law form, and $(\psi_1,\varphi_2)$ defined by \eqref{IE:psi1} and \eqref{IE:varphi2} is a saddle point in implicit-explicit feedback law form. Moreover,
		\begin{equation}\label{SameSP}
			u^{\varphi_1,\psi_2}(\cdot) =u^{\psi_1,\varphi_2}(\cdot) := u^*(\cdot)
		\end{equation}
and
\begin{equation}
J\big( u^*(\cdot) \big) = \frac {\Gamma}{2} +\widetilde J,
\end{equation}
where $\Gamma$ and $\widetilde J$ are given by \eqref{Gamma} and \eqref{tildeJ}, respectively.
	\end{corollary}
	
	\begin{proof}
		Clearly, we only need to prove \eqref{SameSP}. According to \eqref{EI:AE} and \eqref{IE:AE} and also noticing the notation \eqref{Delta}, we have
		\[
		u^{\varphi_1,\psi_2}(\cdot) = -R(\cdot)^{-1} \Big\{ \Theta(\cdot) \mathbb E \big[ x^{\varphi_1,\psi_2}(\cdot)\, \big|\, \mathcal F^0_\cdot \big] +\nu(\cdot) \Big\}
		\]
		and
		\[
		u^{\psi_1,\varphi_2}(\cdot) = -R(\cdot)^{-1} \Big\{ \Theta(\cdot) \mathbb E \big[ x^{\psi_1,\varphi_2}(\cdot)\, \big|\, \mathcal F^0_\cdot \big] +\nu(\cdot) \Big\}.
		\]
		Here, both $x^{\varphi_1,\psi_2}(\cdot)$ and $x^{\psi_1,\varphi_2}(\cdot)$ satisfy the CMF-SDE:
		\[
		\left\{
		\begin{aligned}
			& \mathrm dx^*(t) = \bigg[ A(t)x^*(t) -B(t)R(t)^{-1}\Big\{ \Theta(t) \mathbb E\big[ x^*(t)\, \big|\, \mathcal F^0_t \big]\\
			& \qquad +\nu(t) \Big\} +b(t) \bigg]\, \mathrm dt +C(t)\, \mathrm dW(t) +\bar C(t)\, d\bar W(t),\\
			& x^*(0) =a
		\end{aligned}
		\right.
		\]
(compared with \eqref{Sys}). The uniqueness of the above equation (see Proposition 7.2 in \cite{CZh-16}) implies that $x^{\varphi_1,\psi_2}(\cdot) = x^{\psi_1,\varphi_2}(\cdot)$. Consequently, \eqref{SameSP} holds true.
	\end{proof}

\section{Application to a duopoly problem}\label{Sec:Duop}
	
In this section, we apply the theoretical results to a duopoly competition problem. 
	
\subsection{Formulation of duopoly problem}
	
To formulate the problem, we will modify the optimal control models proposed by Huang et al. \cite{HWW-10} and Wang et al. \cite{WWX-15} to a zero-sum game model. We notice that the dimensions of the involved equations will be changed from one-dimensional in \cite{HWW-10} and \cite{WWX-15} to two-dimensional.
	
Let us consider a market in which only two dominant firms exist. 
Similar to \cite{N-99},
%
%
the liability processes $L^u(\cdot) = (L^u_1(\cdot), L^u_2(\cdot))^\top$ of the two firms are assumed to be governed by
	\[
	-\mathrm d L^u(t) = \big[ B(t)u(t) + b(t) \big]\, \mathrm dt +C(t)\, \mathrm dW(t) +\bar C(t)\, \mathrm d\bar W(t),
	\]
	where the dimensions of $W(\cdot)$ and $\bar W(\cdot)$ are $d=2$ and $\bar d=1$ respectively, $u(t) = (u_1(t), u_2(t))^\top$ takes values in $\mathbb R^2$, 
\[
\begin{aligned}
& B(t) = \begin{pmatrix} b_{11}(t) & b_{12}(t) \\ b_{21}(t) & b_{22}(t) \end{pmatrix}, &&
b(t) = \begin{pmatrix} b_1(t) \\ b_2(t) \end{pmatrix}, \\
& C(t) = \begin{pmatrix} c_{11}(t) & c_{12}(t) \\ c_{21}(t) & c_{22}(t) \end{pmatrix}, &&
\bar C(t) =\begin{pmatrix} \bar c_1(t) \\ \bar c_2(t) \end{pmatrix}.
\end{aligned}
\]
	Here, $b_i(\cdot) >0$ represents the expected liability per unit time of Firm $i$ ($i=1,2$), $C(\cdot)$ and $\bar C(\cdot)$ characterize the liability risk, and $u_i(\cdot)$ denotes the control process of policymaker of Firm $i$ (i.e., Player $i$) and may be the rate of capital injection or withdrawal ($i=1,2$). When $b_{12}(\cdot) \neq 0$ or $b_{21}(\cdot) \neq 0$, it means that the corresponding firm's liability process may be influenced by the competitor firm's policies. 
	
Similar to \cite{HWW-10,N-99,WWX-15}, we assume that the firms are not allowed to invest in the risky assets due to certain supervisory regulations. In other words, the firms only invest in a money account with an interest rate $a(\cdot)>0$. Then, the cash balance processes (i.e., the state process) $x^u(\cdot) = (x^u_1(\cdot), x^u_2(\cdot))^\top$ of the two firms are given by
	\[
	x^u(t) = e^{\int_0^t a(s)\, \mathrm ds} \bigg[ x_0 -\int_0^t e^{-\int_0^s a(r)\, \mathrm dr}\, \mathrm dL^u(s) \bigg],
	\]
	where $x_0 = (x_{01}, x_{02})^\top$ is the initial endowments. Equivalently,
	\begin{equation}\label{E:Sys}
		\left\{
		\begin{aligned}
			& \mathrm dx^u(t) = \big[ a(t)x^u(t) +B(t)u(t) +b(t) \big]\, \mathrm dt\\
			& \qquad\qquad +C(t)\, \mathrm dW(t) +\bar C(t)\, \mathrm d\bar W(t),\\
			& x^u(0) =x_0. 
		\end{aligned}
		\right.
	\end{equation}

In reality, the cash balance process is generally not fully observable (see \cite{HaS-02} for example)
%
%
Instead, the prices of stocks from the two firms can be fully observed. Similar to \cite{BHM-07,HWW-10,WWX-15},
%
%
a linear factor model is adopted to characterize the price process of Firm $i$'s stock as follows ($i=1,2$):
	\[
	\left\{
	\begin{aligned}
		& \frac{\mathrm dS^u_i(t)}{S^u_i(t)} = \big[ f_i(t)x^u_i(t) +\bar h_i(t)\big ]\, \mathrm dt +K_i(t)\, \mathrm dW(t),\\
		& S^{u}_i(0) =S_{i0},
	\end{aligned}
	\right.
	\]
	where the cash balance process $x^u_i(\cdot)$ is the underlying factor and $K_i(t) = (k_{i1}(t), k_{i2}(t))$ presents the instantaneous volatilities. 
\[
y^u_i(\cdot) := \log S^u_i(\cdot) -\log S_{i0},\quad i=1,2.
\]
Then, $y^u(\cdot) =(y^u_1(\cdot),y^u_2(\cdot))^\top$ named the observation process satisfies 
	\begin{equation}\label{E:Obs}
		\left\{
		\begin{aligned}
			& \mathrm d y^u(t) = \big[ H(t)x^u(t) +h(t) \big]\, \mathrm dt +K(t)\, \mathrm dW(t),\\
			& y^u(0) =0,
		\end{aligned}
		\right.
	\end{equation}
where
\[
\begin{aligned}
& H(t)= \begin{pmatrix} f_1(t) & 0 \\ 0 & f_2(t) \end{pmatrix}, \\
& h(t)= \begin{pmatrix} h_1(t) \\ h_2(t) \end{pmatrix} =\begin{pmatrix}\bar h_1(t) - \frac 1 2 (k_{11}(t)^2 +k_{12}(t)^2) \\ \bar h_2(t) - \frac 1 2 (k_{21}(t)^2 +k_{22}(t)^2) \end{pmatrix}
\end{aligned}
\]
and
\[
K(t)= \begin{pmatrix} K_1(t) \\ K_2(t) \end{pmatrix} =\begin{pmatrix} k_{11}(t) & k_{12}(t) \\ k_{21}(t) & k_{22}(t) \end{pmatrix}.
\]
	The filtration generated by $y^u(\cdot)$ (equivalently, by $S^u(\cdot)$) reads $\mathbb F^u = \{\mathcal F^u_t\}_{t\in [0,T]}$ with
	\[
	\mathcal F^u_t = \sigma\big\{ S^u(s),\ 0\leq s\leq t \big\} =\sigma\big\{ y^u(s),\ 0\leq s\leq t \big\}.
	\]
	
	Now, an objective functional is given to the two policymakers:
	\begin{equation}\label{E:Cost}
		\begin{aligned}
			& \mathcal J(u(\cdot)) = \frac 1 2 \mathbb E\bigg\{ M_2 \big| x^u_2(T) -m_2 \big|^2\\
			& \quad -M_1 \big| x^u_1(T) -m_1 \big|^2 +\int_0^T \Big[ R_2(t) \big| u_2(t) -r_2(t) \big|^2\\
			& \quad -R_1(t) \big| u_1(t) -r_1(t) \big|^2 \Big]\, \mathrm dt \bigg\},
		\end{aligned}
	\end{equation}
	where, $m_i$ is the pre-set target of Firm $i$'s cash balance at the terminal time $T$, $r_i(\cdot)$ is the benchmark process of Policymaker $i$'s control, $M_i \geq 0$ and $R_i(\cdot)\geq 0$ are the related weighting factors ($i=1,2$). Policymaker $1$ wants to maximize $\mathcal J(\cdot)$, while Policymaker $2$ wants to minimize it, i.e., both policymakers hope that their firm's terminal cash balance is close to their pre-set targets, while the controls they adopt are not far from their benchmarks.

	Next, we will construct saddle points in the form of feedback laws for policymakers to address the above duopoly problem with partial observation.

	\subsection{Solution to duopoly problem}
	
	We begin with rewriting the objective functional \eqref{E:Cost} as
	\[
	\mathcal J(u(\cdot)) = J(u(\cdot)) +\mathring{\mathcal J},
	\]
where
\begin{equation}\label{E:ringJ}
\begin{aligned}
\mathring{\mathcal J} =\ & \frac 1 2 \bigg\{ M_2m_2^2  -M_1m_1^2\\
& +\int_0^T \Big[ R_2(t)r_2(t)^2 -R_1(t)r_1(t)^2 \Big]\, \mathrm dt \bigg\}
\end{aligned}
\end{equation}
is a constant and $J(u(\cdot))$ has the form of \eqref{J} with
\[
\begin{aligned}
& G = \begin{pmatrix} -M_1 & \\ & M_2 \end{pmatrix}, && g =\begin{pmatrix} M_1m_1 \\ -M_2m_2\end{pmatrix},\\
& R(\cdot) = \begin{pmatrix} -R_1(\cdot) & \\ & R_2(\cdot) \end{pmatrix}, && \rho(\cdot) = \begin{pmatrix} R_1(\cdot)r_1(\cdot) \\ -R_2(\cdot)r_2(\cdot) \end{pmatrix},
\end{aligned}
\]
$Q(\cdot) =0$, $S(\cdot) =0$ and $q(\cdot) =0$. Since $\mathring{\mathcal J}$ is a constant, the functionals $\mathcal J$ and $J$ are essentially equivalent.
	
	In order to apply the obtained theoretical results, we introduce the following
	
	\noindent{\bf Assumption (H4)} Let $i=1,2$. All the processes $a(\cdot)$, $B(\cdot)$, $b(\cdot)$, $C(\cdot)$, $\bar C(\cdot)$, $f_i(\cdot)$, $K(\cdot)$, $h(\cdot)$, $R_i(\cdot)$ and $r_i(\cdot)$ are deterministic and bounded. $K(\cdot)^{-1}$ exists and is also bounded. $R_i(\cdot) \gg 0$. Additionally, $M_i\geq 0$ and $m_i$ are constants.
	
	Obviously, under Assumption (H4), Assumption (H1) holds. Moreover, Condition (I), Condition (II) and Condition (I \& II) are the same and also hold true. 
	
Firstly, the Riccati equation \eqref{Riccati} and the ODE \eqref{ODE:p} read
	\begin{equation}\label{E:Riccati}
		\left\{
		\begin{aligned}
			& -\dot P = 2aP -PBR^{-1}B^\top P,\\
			& P(T) =G
		\end{aligned}
		\right.
	\end{equation}
	and
	\begin{equation}\label{E:aux-ODE}
		\left\{
		\begin{aligned}
			& -\dot p =ap +Pb -PBR^{-1}\big( B^\top p +\rho \big),\\
			& p(T) =g,
		\end{aligned}
		\right.
	\end{equation}
respectively.
%

Secondly, under Assumptions (H3) and (H4), the feedback laws \eqref{EI:varphi1}, \eqref{EI:psi2}, \eqref{IE:psi1} and \eqref{IE:varphi2} are reduced to
	\begin{equation}\label{E:Feed-Laws}
		\begin{aligned}
			& \varphi_1(t,x,u_2) = \psi_1(t,x)\\
			& \qquad\qquad = \frac{B_1(t)^\top \big[ P(t)x +p(t) \big]}{R_1(t)} +r_1(t),\\
			& \psi_2(t,x) = \varphi_2(t,x,u_1)\\
			& \qquad\qquad = -\frac{B_2(t)^\top \big[ P(t)x +p(t) \big]}{R_2(t)} +r_2(t),
		\end{aligned}
	\end{equation}
	for any $t\in[0,T]$, $x\in \mathbb R^2$ and $u_1, u_2 \in \mathbb R$. We notice that $R_{21}(\cdot) =0$ in the objective functional \eqref{E:Cost} leads to the special expressions of $\varphi_1$ and $\varphi_2$ that are  independent of $u_2$ and $u_1$ respectively. As a result of Theorem \ref{THM:EI-Sad}, Theorem \ref{THM:IE-Sad} and Corollary \ref{Cor:Same-Sad}, $(\varphi_1,\psi_2) = (\psi_1,\varphi_2)$ defined by \eqref{E:Feed-Laws} provides a saddle point in feedback law form.

	Thirdly, the induced controls are given by $u_1^*(\cdot) = \psi_1 (\cdot, \widehat x^*(\cdot) )$ and $u_2^*(\cdot) = \psi_2 (\cdot, \widehat x^*(\cdot) )$,
	where $\widehat x^*(\cdot)$ is the solution to the following filtering equation:
	\begin{equation}
		\left\{
		\begin{aligned}\label{E:KB}
			& \mathrm d\widehat x^* = \bigg\{ \bigg[ aI +\bigg( \frac{B_1B_1^\top}{R_1} -\frac{B_2B_2^\top}{R_2} \bigg) P \bigg] \widehat x^*\\
			& \qquad +\bigg( \frac{B_1B_1^\top}{R_1} -\frac{B_2B_2^\top}{R_2} \bigg) p\\
			& \qquad +B_1r_1 +B_2r_2 +b\bigg\}\, \mathrm dt\\
			& \qquad + \Big\{ C +\Sigma \big[ K^{-1}H \big]^\top \Big\}\, \mathrm d\widehat W,\\
			& \widehat x^*(0) =x_0
		\end{aligned}
		\right.
	\end{equation}
	(noticing $\Sigma(\cdot)$ is the solution to \eqref{KB:Riccati}).

Finally, the corresponding value of the objective functional is 
\begin{equation}
\mathcal J(u^*(\cdot)) = \frac {\Gamma}{2} +\widetilde J +\mathring{\mathcal J},
\end{equation}
	where the constants $\Gamma$, $\widetilde J$ and $\mathring{\mathcal J}$ are given by \eqref{Gamma}, \eqref{tildeJ} and \eqref{E:ringJ}, respectively.

\subsection{Numerical simulation}
	
We would like to continue doing a bit of numerical simulation to get some intuition.
	
Let $T =M_1 =M_2 =1.0$, $m_1 =m_2 =2.0$, $a(\cdot) =0.03$, $R_1(\cdot)=R_2(\cdot)=r_1(\cdot) =r_2(\cdot) =1.0$,
\[
\begin{aligned}
& B(\cdot) = \begin{pmatrix} 1.0 & 0.5 \\ 0.5 & 1.0 \end{pmatrix}, &&
b(\cdot) = \begin{pmatrix} 0.5 \\ 0.3 \end{pmatrix}, \\
& C(\cdot) = \begin{pmatrix} 1.0 & 0.5 \\ 0.5 & 1.0 \end{pmatrix}, &&
\bar C(\cdot) =\begin{pmatrix} 0.5 \\ 0.3 \end{pmatrix}, \\
& H(\cdot)= \begin{pmatrix} 1.0 & 0 \\ 0 & 1.0 \end{pmatrix}, && K(\cdot) =\begin{pmatrix} 1.0 & 0.5 \\ 0.5 & 1.0 \end{pmatrix}.
\end{aligned}
\]
Under the above setting, it can be directly verified that the coefficients satisfy the assumptions of Proposition A.2 in \cite{TYZh-20}, which ensures the unique solvability of both the Riccati equation \eqref{E:Riccati} and the ODE \eqref{E:aux-ODE}.

Next, we solve the Riccati equation \eqref{E:Riccati} numerically using the Euler method, and plot the corresponding curves for each component of it in Fig.~\ref{fig1}. Similarly, we also plot the curves for $p(\cdot)$ and $\Sigma(\cdot)$ in Fig.~\ref{fig2} and Fig.~\ref{fig3}, respectively. From Fig.~\ref{fig1} and Fig.~\ref{fig3}, it is observed that the curves of $P_{12}(\cdot)$ and $P_{21}(\cdot)$, as well as the curves of $\Sigma_{12}(\cdot)$ and $\Sigma_{21}(\cdot)$, are identical, which intuitively demonstrates the symmetry of $P(\cdot)$ and $\Sigma(\cdot)$.


Moreover, the numerical solution of the filtering equation \eqref{E:KB} can be obtained by virtue of the Euler-Maruyama method. A trajectory of $\widehat{x}^*(\cdot)$ is presented in Fig.~\ref{fig4}. We continue to plot trajectories of the induced controls $u_1^*(\cdot)$ and $u_2^*(\cdot)$ in Fig.~\ref{fig5}. From Fig.~\ref{fig5}, we observe that the trajectories of $u_1^*(\cdot)$ and $u_2^*(\cdot)$ exhibit opposite trends that are almost opposite in terms of their increases and decreases, which intuitively reflects the competitive relationship between the two firms.

	\begin{figure}
		\begin{center}
			\includegraphics[height=6.27cm, width=7cm]{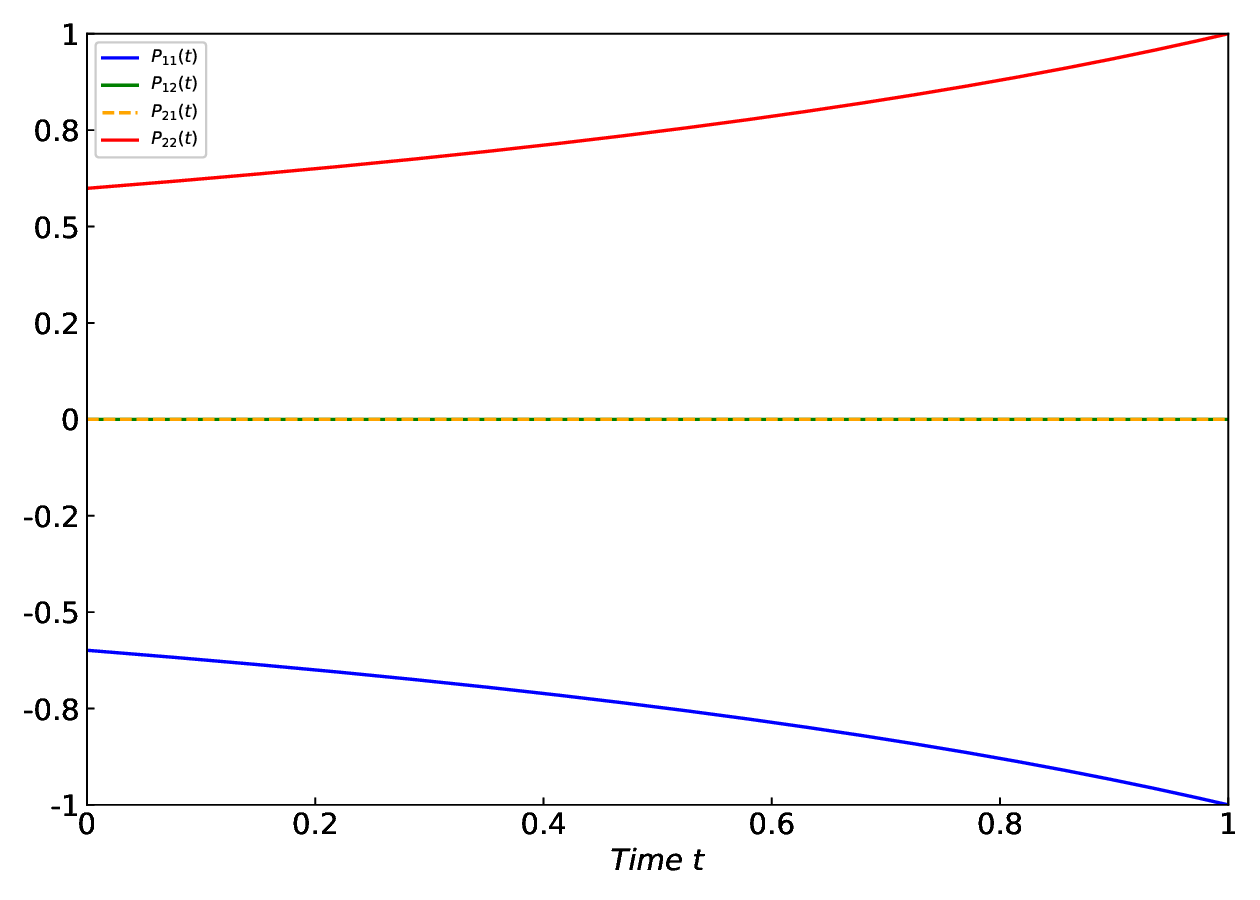}    
			\caption{Numerical solution of $P(\cdot)$.}  
			\label{fig1}                                 
		\end{center}                                 
	\end{figure}
	
	\begin{figure}
		\begin{center}
			\includegraphics[height=6.27cm, width=7cm]{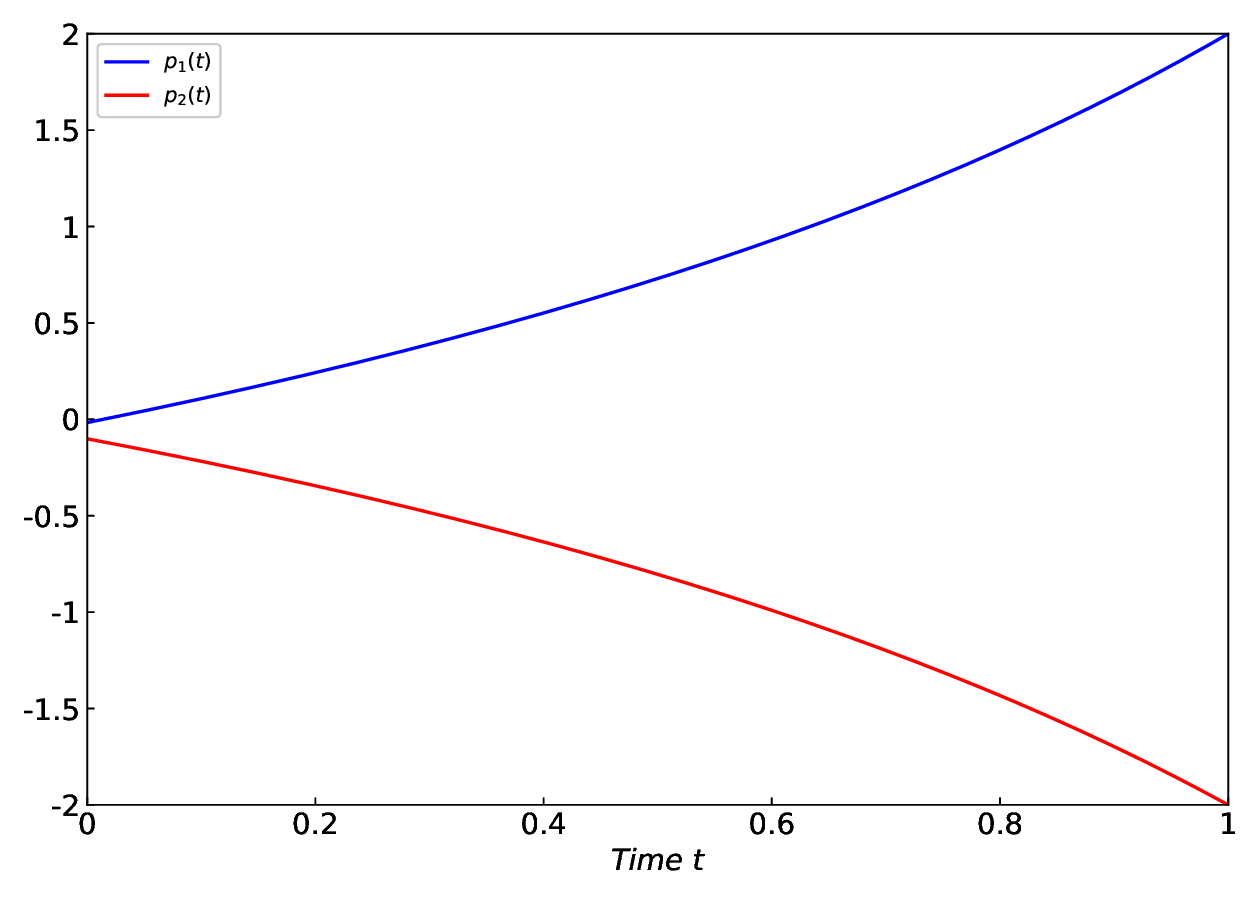}    
			\caption{Numerical solution of $p(\cdot)$.}  
			\label{fig2}                                 
		\end{center}                                 
	\end{figure}
	
	\begin{figure}
		\begin{center}
			\includegraphics[height=6.27cm, width=7cm]{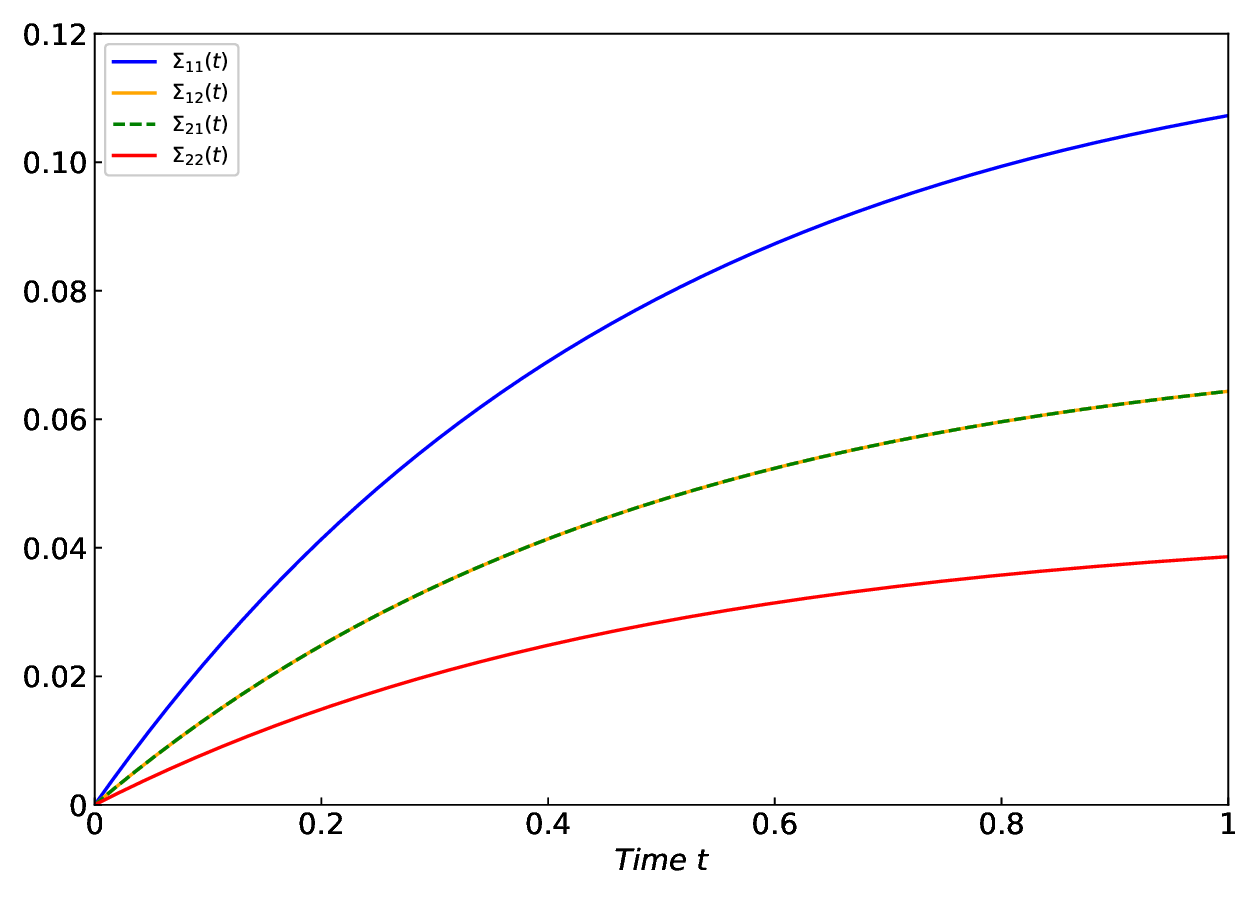}    
			\caption{Numerical solution of $\Sigma(\cdot)$.}  
			\label{fig3}                                 
		\end{center}                                 
	\end{figure}
	
	\begin{figure}
		\begin{center}
			\includegraphics[height=6.27cm, width=7cm]{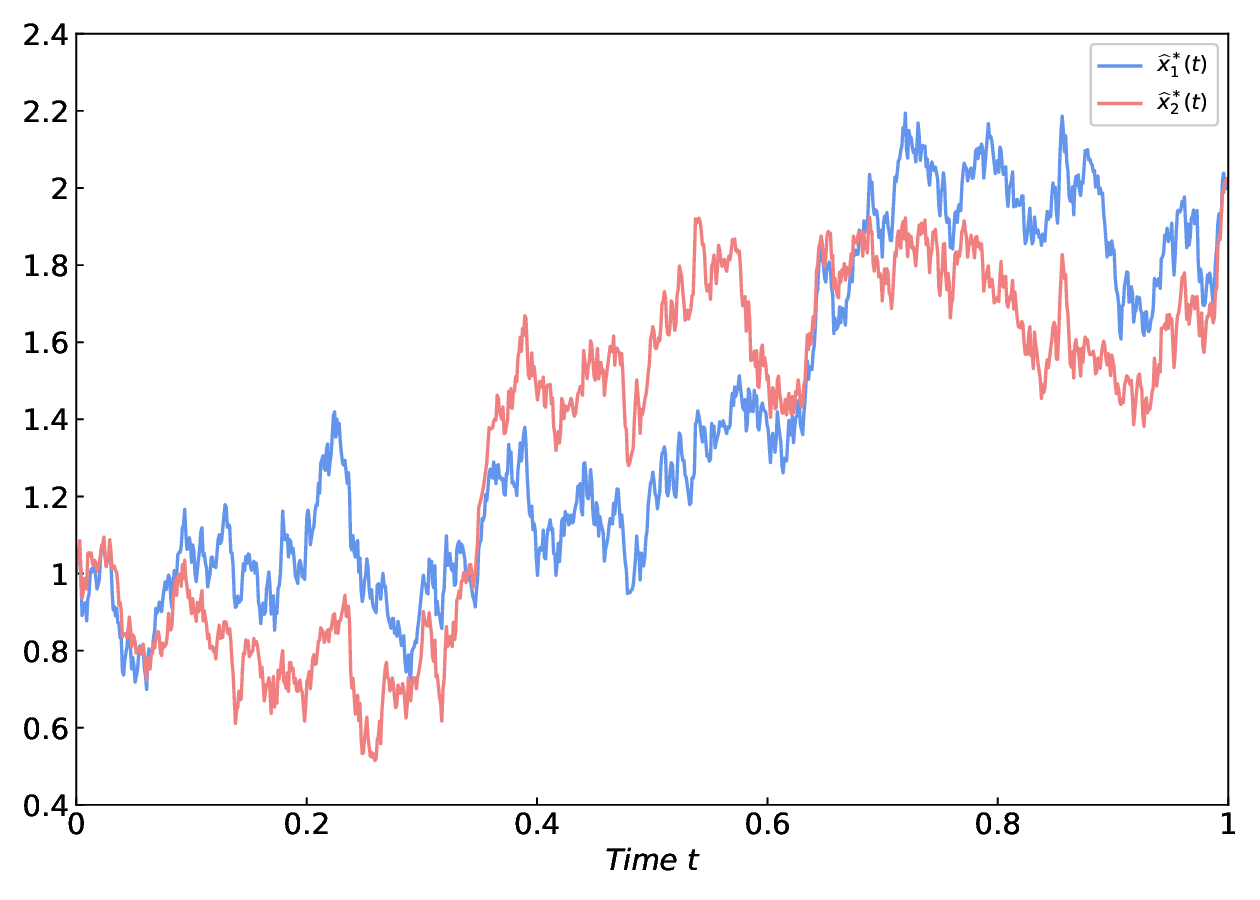}    
			\caption{One trajectory of $\widehat{x}^*(\cdot)$.}  
			\label{fig4}                                 
		\end{center}                                 
	\end{figure}
	
	\begin{figure}
		\begin{center}
			\includegraphics[height=6.27cm, width=7cm]{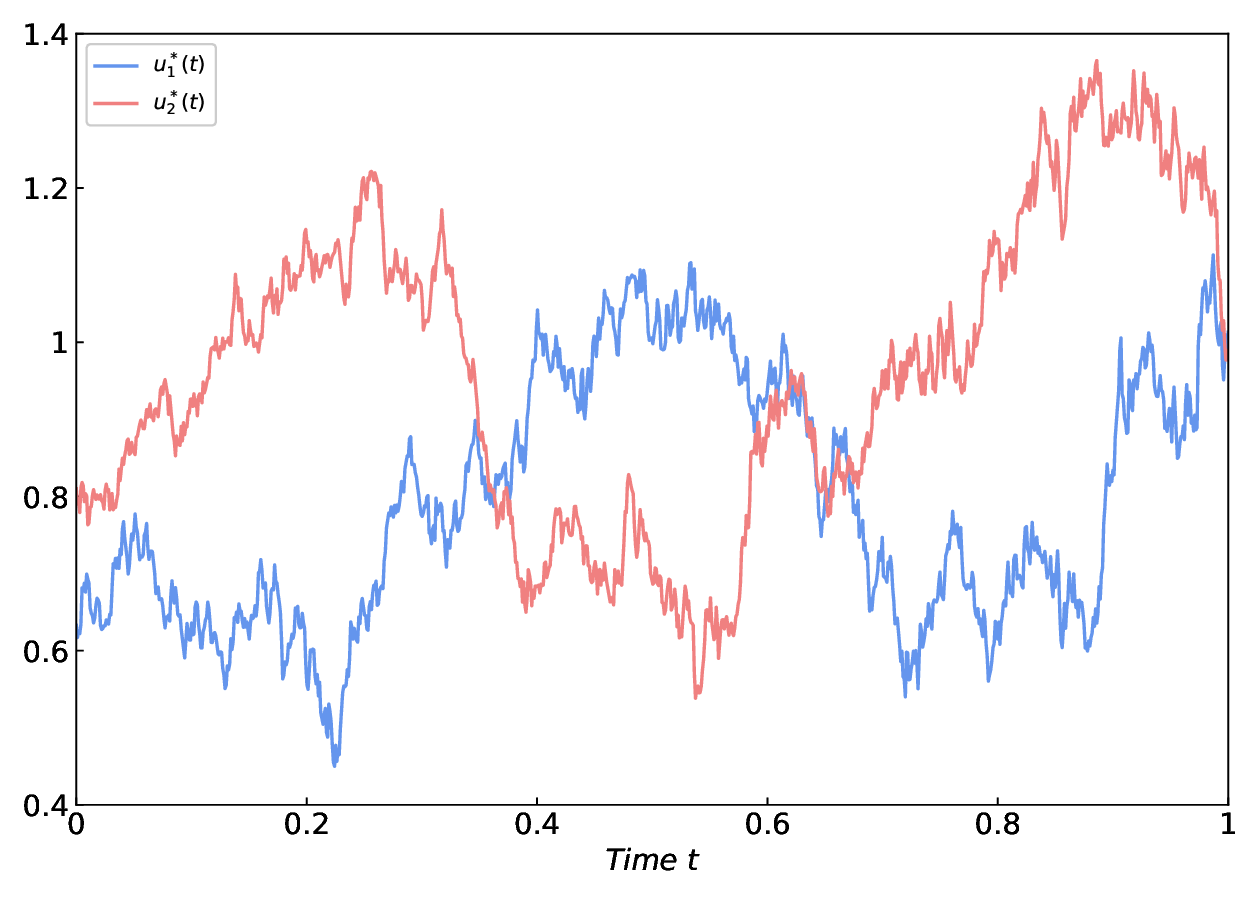}
			\caption{The induced pair of controls $(u_1^*(\cdot), u_2^*(\cdot))$.}  
			\label{fig5}                                 
		\end{center}                                 
	\end{figure}

\section{Conclusion and further work}\label{Sec:Conclu} In this paper, we obtain a feedback saddle point for an LQ zero-sum stochastic differential game when the players share a common observation equation. A natural generalization is to consider the case where the players have different observation equations. However, this raises new difficulties such as nonlinear filtering equations. We hope to report on the progress made in this direction in our future work.

\begin{ack} 
The authors would like to thank the editor, the associate editor, and the anonymous referees for their constructive and insightful comments, which help us improve the quality of this work.                              
This work was supported in part by the Natural Science Foundation of Shandong Province (ZR2024ZD35 and ZR2025MS63) and the National Natural Science Foundation of China (12271304).  
\end{ack}

\end{document}